\documentclass[leqno]{siamltex704}
\usepackage{color,xcolor}
\usepackage{amsmath}
\usepackage{graphicx}
\usepackage{mathrsfs}
\usepackage{float}
\usepackage{amsfonts,amssymb}
\usepackage{dsfont}
\usepackage{pifont}
\usepackage{hyperref}
\usepackage{multirow}
\numberwithin{equation}{section}
\def\3bar{{|\hspace{-.02in}|\hspace{-.02in}|}}

\def\T{{\mathcal{T}}}

\def\v{\varphi}

\def\bv{{\mathbf{v}}}
\def\bn{{\mathbf{n}}}

\newtheorem{example}{\bf Example}[section]
\newtheorem{remark}{Remark}[section]
\newtheorem{algorithm}{Weak Galerkin Algorithm}

\title{A Weak Galerkin Finite Element Scheme for the Biharmonic Equations
by Using Polynomials of Reduced Order}
\author{ Ran Zhang
\thanks{Department of Mathematics, Jilin University, Changchun,
China (zhangran@mail.jlu.edu.cn). The research of Zhang was
supported in part by China Natural National Science
Foundation(11271157, 11371171, 11471141), and by the Program for New
Century Excellent Talents in University of Ministry of Education of
China.} \and Qilong Zhai\thanks{Department of Mathematics, Jilin
University, Changchun, China (diql@mails.jlu.edu.cn).}  }

\begin{document}

\maketitle

\begin{abstract}
A new weak Galerkin (WG) finite element method for solving the
biharmonic equation in two or three dimensional spaces by using
polynomials of reduced order is introduced and analyzed. The WG
method is on the use of weak functions and their weak derivatives
defined as distributions. Weak functions and weak derivatives can be
approximated by polynomials with various degrees. Different
combination of polynomial spaces leads to different WG finite
element methods, which makes WG methods highly flexible and
efficient in practical computation. This paper explores the
possibility of optimal combination of polynomial spaces that
minimize the number of unknowns in the numerical scheme, yet without
compromising the accuracy of the numerical approximation. Error
estimates of optimal order are established for the corresponding WG
approximations in both a discrete $H^2$ norm and the standard $L^2$
norm. In addition, the paper also presents some numerical
experiments to demonstrate the power of the WG method. The numerical
results show a great promise of the robustness, reliability,
flexibility and accuracy of the WG method.
\end{abstract}



\section{Introduction}
This paper will concern with approximating the solution $u$ of the
biharmonic equation
\begin{eqnarray}\label{biharmonic}
\Delta^2 u&=&f, \quad {\rm in} \ \Omega,
\end{eqnarray}
with clamped boundary conditions
\begin{eqnarray}
\label{boundary_value}u&=&g,\quad {\rm on}\ \partial\Omega,
\\
\label{boundary_normal}\frac{\partial u}{\partial
\textbf{n}}&=&\phi,\quad {\rm on}\ \partial\Omega,
\end{eqnarray}
where $\Delta$ is the Laplacian operator, $\Omega$ is a bounded
polygonal or polyhedral domain in $\mathbb{R}^d$ for $d=2, 3$ and
$\textbf{n}$ denotes the outward unit normal vector along
$\partial\Omega$. We assume that $f, g, \phi$ are given,
sufficiently smooth functions.

This problem mainly arises in fluid dynamics where the stream
functions $u$ of incompressible flows are sought and elasticity
theory,  in which the deflection of a thin plate of the clamped
plate bending problem is sought \cite{JS77, Muskhelishvili53,
Roache72}.

Due to the significance of the biharmonic problem, a large number of
methods for discretizing (\ref{biharmonic}) -
(\ref{boundary_normal}) have been proposed. These methods include
dealing with the  biharmonic operator directly, such as discretizing
(\ref{biharmonic})-(\ref{boundary_normal}) on a uniform grid using a
13-point or 25-point direct approximation of the fourth order
differential operator \cite{Bjorstad83, GM79}; mixed methods, that
is, splitting the biharmonic equation into two coupled Poisson
equations \cite{AYB98, BG11, CDG09, AB85, CLL08, CR74, DP01, DGP91,
EG75, Linden85,Heinrichs91,BK00,BK10}. Also there are some other
approaches to the biharmonic problems, like the conformal mapping
methods \cite{CDH97, Pandit08}, integral equations \cite{Mayo84},
orthogonal spline collocation method \cite{Bialecki03} and the fast
multipole methods \cite{GGM92}, etc.

Among these methods, finite element methods are one of the most
widely used technique, which is based on variational formulations of
the equations considered. In fact, the biharmonic equation is also
one of the most important applicable problems of the finite element
methods, cf. \cite{CT65, Fraeijs65, AFS68, Zlamal68,CCQ131,CCQ132}.
The Galerkin methods, discretizing the corresponding variational
form of (\ref{biharmonic}) is given by seeking $u\in H^2(\Omega)$
satisfying
$$
u|_{\partial \Omega}=g, \qquad \frac{\partial u}{\partial
\textbf{n}}|_{\partial \Omega}=\phi
$$
such that
\begin{eqnarray}
\label{variational_form} (\Delta u, \Delta v)=(f,v), \quad \forall
v\in H_0^2(\Omega),
\end{eqnarray}
where $H_0^2(\Omega)$ is the subspace of $H^2(\Omega)$ consisting of
functions with vanishing value and normal derivative on $\partial
\Omega$.

Standard finite element methods for solving (\ref{biharmonic}) -
(\ref{boundary_normal}) based on the variational form
(\ref{variational_form}) with conforming finite element require
rather sophisticated finite elements such as the
21-degrees-of-freedom of Argyris (see \cite{AD76}) or nonconforming
elements of Hermite type. Since the complexity in the construction
for the finite element with high continuous elements, $H^2$
conforming element are seldom used in practice for the biharmonic
problem.  To avoid using of $C^1$-elements, besides the mixed
methods, an alternative approach, nonconforming and discontinuous
Galerkin finite element methods have been developed for solving the
biharmonic equation over the last several decades. Morley element
\cite{Morley68} is a well known nonconforming element for the
biharmonic equation for its simplicity. A $C^0$ interior penalty
method was developed in \cite{BS05,EGHLMT02}. In \cite{MB07}, a
hp-version interior penalty discontinuous Galerkin method was
presented for the biharmonic equation.

Recently a new class of finite element methods, called weak
Galerkin(WG) finite element methods were developed for the
biharmonic equation for its highly flexible and robust properties.
The WG method refers to a numerical scheme for partial differential
equations in which differential operators are approximated by weak
forms as distributions over a set of generalized functions. This
thought was first proposed in \cite{WY1} for a model second order
elliptic problem, and this method was further developed in
\cite{WY2, MWY,WY3}. In \cite{MWY3}, a weak Galerkin method for the
biharmonic equation was derived by using discontinuous functions of
piecewise polynomials on general partitions of polygons or polyhedra
of arbitrary shape. After that, in order to reduce the number of
unknowns, a $C^0$ WG method \cite{MWYZ} was proposed and analyzed.
However, due to the continuity limitation, the $C^0$ WG scheme only
works for the traditional finite partitions, while not arbitrary
polygonal or polyhedral girds as allowed in \cite{MWY3}.

In order to realize the aim that reducing the unknown numbers and
suit for general partitions of polygons or polyhedra of arbitrary
shape at the same time, in this paper we construct a reduction WG
scheme based on the use of a discrete weak Laplacian plus a new
stabilization that is also parameter free. The goal of this paper is
to specify all the details for the reduction WG method for the
biharmonic equations and present the numerical analysis by
presenting a mathematical convergence theory.

An outline of the paper is as follows. In the remainder of the
introduction we shall introduce some preliminaries and notations for
Sobolev spaces. In Section 2 is devoted to the definitions of weak
functions and weak derivatives. The WG finite element schemes for
the biharmonic equation (\ref{biharmonic})-(\ref{boundary_normal})
are presented in Section 3. In Section 4, we establish an optimal
order error estimates for the WG finite element approximation in an
$H^2$ equivalent discrete norm. In Section 5, we shall drive an
error estimate for the WG finite element method in the standard
$L^2$ norm. Section 6 contains the numerical results of the WG
method. The theoretical results are illustrated by these numerical
examples. Finally, we present some technical estimates for
quantities related to the local $L^2$ projections into various
finite element spaces and some approximation properties which are
useful in the convergence analysis in Appendix A.
\medskip

Now let us define some notations. Let $D$ be any open bounded domain
with Lipschitz continuous boundary in $\mathbb{R}^d, d=2, 3$. We use
the standard definition for the Sobloev space $H^s(D)$ and their
associated inner products $(\cdot, \cdot)_{s, D}$, norms
$\|\cdot\|_{s, D}$, and seminorms $|\cdot|_{s, D}$ for any $s\ge 0$.

The space $H^0(D)$ coincides with $L^2(D)$, for which the norm and
the inner product are denoted by $\|\cdot\|_D$ and
$(\cdot,\cdot)_D$, respectively. When $D=\Omega$, we shall drop the
subscript $D$ in the norm and in the inner product notation.

The space $H({\rm div}; D)$  is defined as the set of vector-valued
functions on $D$ which, together with their divergence, are square
integrable; i.e.,
$$
H({\rm div}; D)=\{\textbf{v}: \textbf{v}\in [L^2(D)]^d, \nabla \cdot
\textbf{v}\in L^2(D)\}.
$$
The norm in $H({\rm div}; D)$ is defined by
$$
\|\textbf{v}\|_{H({\rm div}; D)}=(\|\textbf{v}\|^2_D+\|\nabla\cdot
\textbf{v}\|^2_D)^{\frac12}.
$$

\section{Weak Laplacain and Discrete Weak
Laplacian}\label{Section:weakLaplacian}

For the biharmonic equation (\ref{biharmonic}), the underlying
differential operator is the Laplacian $\Delta$. Thus, we shall
first introduce a weak version for the Laplacian operator defined on
a class of discontinuous functions as distributions \cite{MWY3}.

Let $K$ be any polygonal or polyhedral domain with boundary
$\partial K$. A weak function on the region $K$ refers to a function
$v= \{v_0, v_b, \textbf{v}_g\}$ such that $v_0\in L^2(K)$, $v_b\in
L^{2}(\partial K)$, and $\textbf{v}_g\cdot \textbf{n} \in
L^{2}(\partial K)$, where $\textbf{n}$ is the outward unit normal
vector along $\partial K$. Denote by $\mathcal{W}(K)$ the space of
all weak functions on $K$, that is,
\begin{eqnarray}
\label{space_wfunctions} \mathcal{W}(K)=\{v=\{v_0,v_b,
\textbf{v}_g\}: v_0\in L^2(K), v_b, \textbf{v}_g\cdot \textbf{n} \in
L^{2}(\partial K)\}.
\end{eqnarray}

Recall that, for any $v\in \mathcal{W}(K)$, the {\rm weak Laplacian}
of $v=\{v_0, v_b, \textbf{v}_g\}$ is defined as a linear functional
$\Delta_w v$ in the dual space of $H^2(K)$ whose action on each
$\varphi\in H^2(K)$ is given by
\begin{eqnarray}
\label{wLaplacian} (\Delta_w v, \varphi)_K=(v_0, \Delta
\varphi)_K-\langle v_b, \nabla \varphi\cdot
\textbf{n}\rangle_{\partial K}+\langle\textbf{v}_g\cdot
\textbf{n},\varphi\rangle_{\partial K},
\end{eqnarray}
where $(\cdot, \cdot)_K$ stands for the $L^2$-inner product in
$L^2(K)$ and $\langle\cdot, \cdot\rangle_{\partial K}$ is the inner
product in $L^2(\partial K)$.

The Sobolev space $H^2(K)$ can be embedded into the space
$\mathcal{W}(K)$ by an inclusion map $i_\mathcal{W}:
H^2(K)\rightarrow \mathcal{W}(K)$ defined as follows
$$
i_\mathcal{W}(\phi)=\{\phi|_K, \phi|_{\partial K}, (\nabla \phi\cdot
\textbf{n})\textbf{n}|_{\partial K}\}, \quad \phi\in H^2(K).
$$
With the help of the inclusion map $i_\mathcal{W}$, the Sobolev
space $H^2(K)$ can be viewed as a subspace of $\mathcal{W}(K)$ by
identifying each $\phi\in H^2(K)$ with $i_\mathcal{W}(\phi)$.

Analogously, a weak function $v= \{v_0,v_b, \textbf{v}_g\}\in
\mathcal{W}(K)$ is said to be in $H^2(K)$ if it can be identified
with a function $\phi\in H^2(K)$ through the above inclusion map.
Here the first components $v_0$ can be seen as the value of $v$ in
the interior and the second component $v_b$ represents the value of
$v$ on $\partial K$. Denote $\nabla v\cdot \textbf{n}$ by $v_n$,
then the third component $\textbf{v}_g$ represents $(\nabla v\cdot
\textbf{n})\textbf{n}|_{\partial K}=v_n \textbf{n}$. Obviously,
$\textbf{v}_g\cdot\textbf{n}= \nabla v\cdot\textbf{n}$. Note that if
$v\not\in H^2(K)$, then $v_b$ and $\textbf{v}_g$ may not necessarily
be related to the trace of $v_0$ and $(\nabla v_0\cdot
\textbf{n})\textbf{n}$ on $\partial K$, respectively.

For $v\in H^2(K)$, from integration by parts we have
\begin{eqnarray*}
(\Delta_w v, \varphi)_K&=&(v, \Delta \varphi)_K-\langle v, \nabla
\varphi\cdot \textbf{n}\rangle_{\partial K}+\langle\nabla v\cdot
\textbf{n},\varphi\rangle_{\partial K}
\\
&=& (v_0, \Delta \varphi)_K-\langle v_b, \nabla \varphi\cdot
\textbf{n}\rangle_{\partial K}+\langle\textbf{v}_g\cdot
\textbf{n},\varphi\rangle_{\partial K}.
\end{eqnarray*}
Thus the weak Laplacian is identical with the strong Laplacian,
i.e.,
$$
\Delta_w i_\mathcal{W}(v)=\Delta v
$$
for smooth functions in $H^2(K)$.

For numerical implementation purpose, we define a discrete version
of the weak Laplacain operator by approximating $\Delta_w$ in
polynomial subspaces of the dual of $H^2(K)$. To this end, for any
non-negative integer $r\ge 0$, let $P_r(K)$ be the set of
polynomials on $K$ with degree no more than $r$.

\begin{definition}{\rm (\cite{MWY3})}
\label{DweakLaplacian} A discrete weak Laplacian operator, denoted
by $\Delta_{w,r,K}$, is defined as the unique polynomial $\Delta_{w,
r, K} v\in P_r(K)$ satisfying
\begin{eqnarray}
\label{Discrete_wLaplacian} \qquad(\Delta_{w,r,K} v,
\varphi)_K=(v_0, \Delta \varphi)_K-\langle v_b, \nabla \varphi\cdot
\textbf{n}\rangle_{\partial K}+\langle\textbf{\rm \bf v}_n\cdot
\textbf{n},\varphi\rangle_{\partial K}, \quad \forall \varphi\in
P_r(K).
\end{eqnarray}
\end{definition}

From the integration by parts, we have
$$
(v_0, \Delta \varphi)_K = (\Delta v_0, \varphi)_K+\langle v_0,
\nabla \varphi\cdot \textbf{n}\rangle_{\partial K}-\langle\nabla
v_0\cdot \textbf{n},\varphi\rangle_{\partial K}.
$$
Substituting the above identity into (\ref{Discrete_wLaplacian})
yields
\begin{eqnarray}
\label{Discrete_wLaplacian-useful} \qquad(\Delta_{w,r,K} v,
\varphi)_K-(\Delta v_0, \varphi)_K=\langle v_0-v_b, \nabla
\varphi\cdot \textbf{n}\rangle_{\partial K}-\langle(\nabla
v_0-\textbf{v}_g)\cdot \textbf{n},\varphi\rangle_{\partial K},
\end{eqnarray}
for all $\varphi\in P_r(K)$.

\section{Weak Galerkin Finite Element Scheme}\label{Section:WG-Scheme}

Let $\mathcal{T}_h$ be a partition of the domain $\Omega$ into
polygons in 2D or polyhedra in 3D. Assume that $\mathcal{T}_h$ is
shape regular in the sense as defined in \cite{WY2}. Denote by
$\mathcal{E}_h$ the set of all edges or flat faces in
$\mathcal{T}_h$, and let
$\mathcal{E}_h^0=\mathcal{E}_h\setminus\partial \Omega$ be the set
of all interior edges or flat faces.

Since $v_n$ represents $\nabla v\cdot \textbf{n}$, then $v_n$ is
naturally dependent on $\textbf{n}$. To ensure a single valued
function $v_n$ on $e\in \mathcal{E}_h$, we introduce a set of normal
directions on $\mathcal{E}_h$ as follows
\begin{eqnarray}\label{normal_directions}
\mathcal{N}_h=\{\textbf{n}_e: \textbf{n}_e {\rm \ is\  unit \ and \
normal\  to\ } e, \ e\in \mathcal{E}_h\}.
\end{eqnarray}
For any given integer $k\ge 2, T\in \mathcal{T}_h$, denote by
$\mathcal{W}_k(T)$ the discrete weak function space given by
\begin{eqnarray}\label{Discrete_wfunctionSpace}
\mathcal{W}_k(T)=\{\{v_0,v_b, v_n \textbf{n}_e\}: v_0\in P_k(T),
v_b, v_n\in P_{k-1}(e),  e\subset \partial T\}.
\end{eqnarray}
By patching $\mathcal{W}_k(T)$ over all the elements $T\in
\mathcal{T}_h$ through a common value on the interface
$\mathcal{E}_h^0$, we arrive at a weak finite element space $V_h$
given by
$$
V_h=\{\{v_0, v_b, v_n \textbf{n}_e \}: \{v_0,v_b, v_n
\textbf{n}_e\}\big|_T\in \mathcal{W}_k(T), \quad \forall T\in
\mathcal{T}_h\}.
$$
Denote by $V_h^0$ the subspace of $V_h$ constituting discrete weak
functions with vanishing traces; i.e.,
$$
V_h^0=\{\{v_0, v_b, v_n \textbf{n}_e \}: \{v_0,v_b, v_n
\textbf{n}_e\}\in V_h, v_b|_e=0, v_n|_e=0, \quad e\in \partial T\cap
\partial \Omega\}.
$$
Denote by $\Lambda_h$ the trace of $V_h$ on $\partial \Omega$ from
the component $v_b$. It is obvious that $\Lambda_h$ consists of
piecewise polynomials of degree $k-1$. Similarly, denote by
$\Upsilon_h$ the trace of $V_h$ from the component of $v_n$ as
piecewise polynomials of degree $k-1$. Denote by $\Delta_{w, k-2}$
the discrete weak Laplacian operator on the finite element space
$V_h$ computed by using (\ref{Discrete_wLaplacian}) on each element
$T$ for $k\ge 2$, that is,
\begin{eqnarray}\label{Discrete_wLaplacianoperator}
(\Delta_{w, k-2} v)|_T=\Delta_{w, k-2, T}(v|_T)\quad \forall v\in
V_h.
\end{eqnarray}
For simplicity, we shall drop the subscript $k-2$ in the notation
$\Delta_{w, k-2}$ for the discrete weak Laplacian operator. We also
introduce the following notation
$$
(\Delta_{w}v,\Delta_{w}w)_h=\sum_{T\in \mathcal{T}_h}
(\Delta_{w}v,\Delta_{w}w)_T.
$$

For each element $T\in \mathcal{T}_h$, denote by $Q_0$ the $L^2$
projection onto $P_k(T)$, $k\ge 2$. For each edge/face $e\subset
\partial T$, denote by $Q_b$ the $L^2$ projection onto
$P_{k-1}(e)$. Now for any $u\in H^2(\Omega)$, we shall combine these
two projections together to define a projection into the finite
element space $V_h$ such that on the element $T$
$$
Q_hu=\{Q_0 u, Q_b u, (Q_b (\nabla u\cdot
\textbf{n}_e))\textbf{n}_e\}.
$$

\begin{theorem}\label{Thm:commutative-property}
Let $\mathbb{Q}_h$ be the local $L^2$ projection onto $P_{k-2}$.
Then the following commutative diagram holds true on each element
$T\in\T_h$:
\begin{eqnarray}\label{Commutative}
\Delta_w Q_h u=\mathbb{Q}_h \Delta u,\qquad \forall u\in H^2(T).
\end{eqnarray}
\end{theorem}

\begin{proof}
For any $\phi\in P_{k-2}(T)$, from the definition of the discrete
weak Laplacian and the $L^2$ projection
\begin{eqnarray*}
(\Delta_w Q_h u, \phi)_T &=&(Q_0 u,\Delta \phi)_T-\langle Q_b u,
\nabla \phi\cdot \textbf{n}\rangle_{\partial T}+\langle Q_b(\nabla
u\cdot \textbf{n}_e)\textbf{n}_e\cdot \textbf{n},
\phi\rangle_{\partial T}
\\
&=& (u, \Delta \phi)_T-\langle u, \nabla \phi\cdot
\textbf{n}\rangle_{\partial T}+\langle \nabla u\cdot \textbf{n},
\phi\rangle_{\partial T}
\\
&=&(\Delta u, \phi)_T=(\mathbb{Q}_h \Delta u, \phi),
\end{eqnarray*}
which implies (\ref{Commutative}).
\end{proof}

The commutative property (\ref{Commutative}) indicates that the
discrete weak Laplacian of the $L^2$ projection of $u$ is a good
approximation of the Laplacian of $u$ in the classical sense. This
is a good property of the discrete weak Laplacian in application to
algorithm and analysis.

For any $u_h=\{u_0, u_b, u_n \textbf{n}_e\}$ and $v=\{v_0, v_b, v_n
\textbf{n}_e \}$ in $V_h$, we introduce a bilinear form as follows
\begin{eqnarray*}
s(u_h, v)&=& \displaystyle\sum_{T\in \mathcal{T}_h}h_T^{-1} \langle
\nabla u_0\cdot \textbf{n}_e-u_n, \nabla v_0\cdot \textbf{n}_e-v_n
\rangle_{\partial T}
\\
&&+\sum_{T\in \mathcal{T}_h}h_T^{-3}\langle Q_b u_0-u_b, Q_b v_0-v_b
\rangle_{\partial T}.
\end{eqnarray*}

\begin{algorithm}
Find $u_h = \{u_0, u_b, u_n\textbf{n}_e\}\in V_h$ satisfying $u_b =
Q_b g$ and $u_n = Q_{b}\phi$ on $\partial \Omega$ and the following
equation:
\begin{eqnarray}\label{WGalerkin_Algorithm}
(\Delta_w u_h, \Delta_w v)_h+s(u_h, v)=(f, v_0), \ \forall v=\{v_0,
v_b, v_n\textbf{n}_e\}\in V_h^0.
\end{eqnarray}
\end{algorithm}

\begin{lemma}
\label{Lemma:Lemma4.1} For any $v \in V_h^0$, let $ |\!|\!| v
|\!|\!| $ be given by
\begin{eqnarray}\label{trinorm}
|\!|\!| v |\!|\!| ^2=(\Delta_w v, \Delta_w v)_h+s(v, v).
\end{eqnarray}
Then, $ |\!|\!| \cdot |\!|\!| $ defines a norm in the linear space
$V_h^0$.
\end{lemma}

\begin{proof}
For simplicity, we shall only prove the positivity property for $
|\!|\!| \cdot |\!|\!| $. Assume that $ |\!|\!| v |\!|\!|  = 0$ for
some $v \in V_h^0$. It follows that $\Delta_w v=0$ on each element
T, $Q_b v_0=v_b$ and $\nabla v_0\cdot\textbf{n}_e = v_n$ on each
edge $\partial T$. We claim that $\Delta v_0 = 0$ holds true locally
on each element T. To this end, for any $\varphi \in P_{k-2}(T)$ we
use $\Delta_w v=0$ and the identify
(\ref{Discrete_wLaplacian-useful}) to obtain
\begin{eqnarray}\label{unique1}
0 &=& (\Delta_w v, \varphi)_T\\
\nonumber&=& (\Delta v_0, \varphi)_T+\langle Q_b v_0-v_b, \nabla
\varphi \cdot \textbf{n}\rangle_{\partial T}
+\langle v_n\textbf{n}_e \cdot \textbf{n}-\nabla v_0\cdot \textbf{n}, \varphi \rangle_{\partial T} \\
\nonumber&=& (\Delta v_0, \varphi)_T,
\end{eqnarray}
where we have used the fact that $Q_b v_0-v_b=0$ and
$$
v_n \textbf{n}_e\cdot \textbf{n}- \nabla v_0\cdot \textbf{n}=\pm
(v_n-\nabla v_0\cdot \textbf{n}_e) = 0
$$
in the last equality. The identity (\ref{unique1}) implies that
$\Delta v_0=0$ holds true locally on each element $T$.

Next, we claim that $\nabla v_0 = 0$ also holds true locally on each
element $T$. For this purpose, for any $\phi \in P_{k}(T)$, we
utilize the Gauss formula to obtain
\begin{eqnarray}\label{unique2}
(\nabla v_0, \nabla \phi)_{T}&=&-(\Delta v_0, \phi)_T +\langle
\nabla v_0\cdot \textbf{n}, \phi\rangle_{\partial T} = \langle
\nabla v_0\cdot \textbf{n}, \phi\rangle_{\partial T}.
\end{eqnarray}
By letting $\phi=v_0$ on each element $T$ and summing over all $T$
we obtain
\begin{eqnarray}\label{unique3}
\sum_{T\in \mathcal{T}_h}(\nabla v_0, \nabla v_0)_{T}=\sum_{T\in
\mathcal{T}_h}\langle  \nabla v_0\cdot \textbf{n},
v_0\rangle_{\partial T}.
\end{eqnarray}

For two elements $T_1, T_2\in \mathcal{T}_h$, which share $e\in
\mathcal{E}_h\setminus\partial \Omega$ as a common edge, denote
$v_0^1, v_0^2$ the values of $v$ in the interior of $T_1, T_2$,
respectively. It follows from $Q_b v_0^1=Q_b v_0^2=v_b$ on edge $e$
and the fact $\nabla v_0\cdot\bn_e=v_n\in P_{k-1}(e)$ that
\begin{eqnarray*}
&&\langle  \nabla v_0^1\cdot \textbf{n}_{T_1},
v_0^1\rangle_{e}+\langle \nabla v_0^2\cdot \textbf{n}_{T_2},
v_0^2\rangle_{e}\\
&=&\pm\langle v_n, v_0^1-v_0^2\rangle_{e}=\pm\langle v_n,
Q_bv_0^1-Q_bv_0^2\rangle_{e}=0,
\end{eqnarray*}
where $\textbf{n}_{T_1}, \textbf{n}_{T_2}$ denote the outward unit
normal vectors on $e$ according to elements $T_1, T_2$,
respectively. This, together with $\nabla v_0\cdot \textbf{n}=v_n=0$
on the boundary edge $e\in \mathcal{E}_h\cap \partial \Omega$
implies
$$
\sum_{T\in \mathcal{T}_h}\langle  \nabla v_0\cdot \textbf{n},
v_0\rangle_{\partial T}=0.
$$
It follows from equation (\ref{unique3}) that $\|\nabla v_0\|_T=0$
on each element $T$. Thus, $v_0= const$ locally on each element and
is then continuous across each interior edge $e$ as
$$
v_0|_e = Q_b v_0 = v_b.
$$
The boundary condition of $v_b=0$ then implies that $v\equiv 0$ on
$\Omega$, which completes the proof.
\end{proof}

\begin{lemma}\label{Lemma:Lemma4.2}
The weak Galerkin finite element scheme (\ref{WGalerkin_Algorithm})
has a unique solution.
\end{lemma}

\begin{proof} Assume $u_h^{(1)}$ and  $u_h^{(2)}$ are two solutions of the WG
finite element scheme (\ref{WGalerkin_Algorithm}). It is obvious
that the difference $\rho_h =  u_h^{(1)}- u_h^{(2)}$ is a finite
element function in $V_h^0$ satisfying
\begin{eqnarray}\label{unique4}
(\Delta_w \rho_h, \Delta_w v)_h+s(\rho_h, v)=0, \ \forall v\in
V_h^0.
\end{eqnarray}
By letting $v=\rho_h$ in above equation (\ref{unique4}) we obtain
the following indentity
\begin{eqnarray*}
(\Delta_w \rho_h, \Delta_w \rho_h)_h+s(\rho_h, \rho_h)=0.
\end{eqnarray*}
It follows from Lemma \ref{Lemma:Lemma4.1} that $\rho_h\equiv 0$,
which shows that $u_h^{(1)}=u_h^{(2)}$. This completes the proof.
\end{proof}

\section{An Error Estimate}\label{Section:H2ErrorEstimate}
The goal of this section is to establish an error estimate for the
WG-FEM solution $u_h$ arising from (\ref{WGalerkin_Algorithm}).

First of all, let us derive an error equation for the WG finite
element solution obtained from (\ref{WGalerkin_Algorithm}). This
error equation is critical in convergence analysis.

\begin{lemma} Let $u$ and $u_h \in V_h$ be the
solution of (\ref{biharmonic})-(\ref{boundary_normal}) and
(\ref{WGalerkin_Algorithm}), respectively. Denote by
$$e_h=Q_hu-u_h$$
the error function between the $L^2$ projection of $u$ and its weak
Galerkin finite element solution. Then the error function $e_h$
satisfies the following equation
\begin{eqnarray}\label{error equation}
(\Delta_{\omega} e_h,\Delta_{\omega} v)_h + s(e_h,v)=\ell_u(v)
\end{eqnarray}
for all $v\in V^0_h$. Here
\begin{eqnarray}\label{error equation-ell}
\ell_u(v) &= &\sum_{T\in\mathcal{T}_h}\langle  \Delta
u-\mathbb{Q}_h\Delta u, \nabla v_0\cdot \mathbf{n}
-v_n\mathbf{n}_e\cdot \mathbf{n}\rangle_{\partial T}\\
\nonumber&& -\sum_{T\in\mathcal{T}_h}\langle  \nabla(\Delta
u-\mathbb{Q}_h\Delta u)\cdot \mathbf{n}, v_0-v_b\rangle_{\partial T}
+s(Q_hu,v).
\end{eqnarray}
\end{lemma}

\begin{proof} Using (\ref{Discrete_wLaplacian-useful}) with $\varphi=\Delta_{\omega}Q_hu = \mathbb{Q}_h\Delta u$
we obtain
\begin{eqnarray*}
&&(\Delta_{\omega}Q_hu,\Delta_{\omega}v)_T\\
&=& (\Delta v_0,\mathbb{Q}_h\Delta u)_T + \langle
v_0-v_b,\nabla(\mathbb{Q}_h\Delta u)\cdot
\mathbf{n}\rangle_{\partial T}
-\langle  (\nabla v_0-v_n\mathbf{n}_e)\cdot \mathbf{n},\mathbb{Q}_h\Delta u\rangle_{\partial T}\\
&=& (\Delta u,\Delta v_0)_T + \langle
v_0-v_b,\nabla(\mathbb{Q}_h\Delta u)\cdot
\mathbf{n}\rangle_{\partial T}-\langle  (\nabla
v_0-v_n\mathbf{n}_e)\cdot \mathbf{n},\mathbb{Q}_h\Delta
u\rangle_{\partial T} ,
\end{eqnarray*}
which implies that
\begin{eqnarray}\label{EQN:7.2}
(\Delta u,\Delta v_0)_T &=& (\Delta_{\omega}Q_hu,\Delta_\omega v)_T-\langle  v_0-v_b,\nabla(\mathbb{Q}_h\Delta u)\cdot \mathbf{n}\rangle_{\partial T}\\
 &&+ \langle  (\nabla v_0-v_n\mathbf{n}_e)\cdot \mathbf{n},\mathbb{Q}_h\Delta u\rangle_{\partial T}
 .\nonumber
\end{eqnarray}
Next, it follows from the integration by parts that
\begin{eqnarray*}
(\Delta u, \Delta v_0)_T =(\Delta^2u,v_0)_T +\langle  \Delta
u,\nabla v_0\cdot \mathbf{n}\rangle_{\partial T}-\langle
\nabla(\Delta u)\cdot \mathbf{n},v_0\rangle_{\partial T}.
\end{eqnarray*}
By summing over all $T$ and then using the identity
$(\Delta^2u,v_0)=(f, v_0)$ we arrive at
\begin{eqnarray*}
\sum_{T\in\mathcal{T}_h}(\Delta u,\Delta v_0)_T&=&(f,
v_0)+\sum_{T\in\mathcal{T}_h}\langle  \Delta u, \nabla v_0\cdot
\mathbf{n}-v_n\mathbf{n}_e\cdot \mathbf{n}\rangle_{\partial
T} \\
&&-\sum_{T\in\mathcal{T}_h}\langle  \nabla(\Delta u)\cdot
\textbf{n}, v_0-v_b\rangle_{\partial T},
\end{eqnarray*}
where we have used the fact that $v_n$ and $v_b$ vanish on the
boundary of the domain. Combining the above equation with
(\ref{EQN:7.2}) yields
\begin{eqnarray*}
(\Delta_{\omega}Q_hu,\Delta_{\omega}v)_h &=& (f, v_0)
+\sum_{T\in\mathcal{T}_h}\langle  \Delta u-\mathbb{Q}_h\Delta u, (\nabla v_0-v_n\mathbf{n}_e)\cdot \mathbf{n}\rangle_{\partial T}\\
&&-\sum_{T\in\mathcal{T}_h}\langle  \nabla(\Delta
u-\mathbb{Q}_h\Delta u)\cdot \mathbf{n},v_0-v_b\rangle_{\partial T}.
\end{eqnarray*}
Adding $s(Q_hu, v)$ to both sides of the above equation gives
\begin{eqnarray}\label{EQN:7.3}
&&(\Delta_{\omega}Q_hu,\Delta_{\omega}v)_h + s(Q_hu,v)
\\
\nonumber&=&(f, v_0)
+\sum_{T\in\mathcal{T}_h}\langle  \Delta u-\mathbb{Q}_h\Delta u, (\nabla v_0-v_n\mathbf{n}_e)\cdot \mathbf{n}\rangle_{\partial T}\\
\nonumber&&-\sum_{T\in\mathcal{T}_h}\langle  \nabla(\Delta
u-\mathbb{Q}_h\Delta u)\cdot \mathbf{n}, v_0-v_b\rangle_{\partial
T}+s(Q_hu,v).
\end{eqnarray}
Subtracting (\ref{WGalerkin_Algorithm}) from (\ref{EQN:7.3}) leads
to the following error equation
\begin{eqnarray*}
(\Delta_{\omega}e_h,\Delta_{\omega}v)_h+s(e_h,v)&=&
\sum_{T\in\mathcal{T}_h}\langle  \Delta u-\mathbb{Q}_h\Delta u,(\nabla v_0-v_n\mathbf{n}_e)\cdot \mathbf{n}\rangle_{\partial T}\\
&&-\sum_{T\in\mathcal{T}_h}\langle  \nabla(\Delta
u-\mathbb{Q}_h\Delta u)\cdot \mathbf{n}, v_0-v_b\rangle_{\partial
T}+s(Q_hu,v)
\end{eqnarray*}
for all $v\in V^0_h$. This completes the derivation of (\ref{error
equation}).
\end{proof}

The following Theorem presents an optimal order error estimate for
the error function $e_h$ in the trip-bar norm. We believe this
tripe-bar norm provides a discrete analogue of the usual $H^2$-norm.

\begin{theorem}\label{theorem:theorem7.3}Let $u_h\in V_h$ be the weak Galerkin finite element solution
arising from (\ref{WGalerkin_Algorithm}) with finite element
functions of order $k\geq2$. Assume that the exact solution of
(\ref{biharmonic})-(\ref{boundary_normal}) is sufficiently regular
such that $u\in H^{k+2}(\Omega)$. Then, there exists a constant $C$
such that
\begin{eqnarray}\label{error_estimate0}
 \3bar u_h-Q_hu \3bar \leq Ch^{k-1}\ \|u\|_{k+2}.
\end{eqnarray}
The above estimate is of optimal order in terms of the meshsize $h$,
but not in the regularity assumption on the exact solution of the
biharmonic equation.
\end{theorem}

\begin{proof} By letting $v =e_h$ in the error equation (\ref{error equation}), we
have
\begin{eqnarray}\label{error_estimate1}
\3bar e_h \3bar^2 = \ell(e_h),
\end{eqnarray}
where
\begin{eqnarray}\label{error_estimate1-ell}
 \ell(e_h) &=&\sum_{T\in\mathcal{T}_h}\langle  \Delta u-\mathbb{Q}_h\Delta u, (\nabla e_0-e_n\mathbf{n}_e)\cdot \mathbf{n}\rangle_{\partial T} \\
\nonumber&&-\sum_{T\in\mathcal{T}_h}\langle  \nabla(\Delta u-\mathbb{Q}_h\Delta u)\cdot \mathbf{n}, e_0-e_b\rangle_{\partial T} \\
\nonumber&&+\sum_{T\in\mathcal{T}_h}h_T^{-1}\langle  \nabla
Q_0u\cdot \mathbf{n}_e-Q_b (\nabla u\cdot \mathbf{n}_e), \nabla e_0
\cdot  \mathbf{n}_e-e_n\rangle_{\partial T} \\
\nonumber&&+\sum_{T\in\mathcal{T}_h}h_T^{-3}\langle  Q_b
Q_0u-Q_bu,Q_b e_0-e_b\rangle_{\partial T}.
\end{eqnarray}

The rest of the proof shall estimate each of the terms on the
right-hand side of (\ref{error_estimate1-ell}). For the first term,
we use the Cauchy-Schwarz inequality and the estimates
(\ref{ineqn3}) and (\ref{ineqn4}) in Lemma \ref{Lemma:Lemma5.2} (see
Appendix A) with $m = k$ to obtain
\begin{eqnarray}\label{error_estimate2}
&&\left|\sum_{T\in\mathcal{T}_h}\langle  \Delta u-\mathbb{Q}_h\Delta
u,
(\nabla e_0-e_n\mathbf{n}_e)\cdot \mathbf{n}\rangle_{\partial T}\right|\\
\nonumber&\leq &\left(\sum_{T\in\mathcal{T}_h}h_T\|\Delta
u-\mathbb{Q}_h\Delta u\|^2_{\partial T}\right)^{\frac{1}{2}}
\left(\sum_{T\in\mathcal{T}_h}h_T^{-1}\|\nabla e_0\cdot \mathbf{n}_e-e_n\|^2_{\partial T}\right)^{\frac{1}{2}}\\
\nonumber&\leq &Ch^{k-1}\|u\|_{k+1} |\!|\!| e_h |\!|\!|.
\end{eqnarray}
For the second term, using Lemma \ref{Lemma:Lemma5.2},
\ref{DiscretePoincareinequality} and \ref{Lemma:Lemma6.5} we obtain
\begin{eqnarray}\label{error_estimate3}
&&\left|\sum_{T\in\mathcal{T}_h}\langle  \nabla(\Delta u-\mathbb{Q}_h\Delta u)\cdot \mathbf{n}, e_0-e_b\rangle_{\partial T}\right|\\
\nonumber&\leq& \left|\sum_{T\in\mathcal{T}_h}\langle  \nabla(\Delta
u-\mathbb{Q}_h\Delta u)
\cdot \mathbf{n}, Q_be_0-e_b\rangle_{\partial T}\right|\\
\nonumber&&+\left|\sum_{T\in\mathcal{T}_h}\langle  \nabla(\Delta
u-\mathbb{Q}_h\Delta u)\cdot \mathbf{n}, e_0-Q_be_0\rangle_{\partial
T}\right|
\\
\nonumber&=& \left|\sum_{T\in\mathcal{T}_h}\langle  \nabla(\Delta
u-\mathbb{Q}_h\Delta u)
\cdot \mathbf{n}, Q_be_0-e_b\rangle_{\partial T}\right|\\
\nonumber&&+\left|\sum_{T\in\mathcal{T}_h}\langle  (\nabla( \Delta
u)-Q_b(\nabla(\Delta u)))\cdot \mathbf{n},
e_0-Q_be_0\rangle_{\partial T}\right|
\\
\nonumber&\leq& \left(\sum_{T\in\mathcal{T}_h}h^3_T\|\nabla( \Delta
u-\mathbb{Q}_h\Delta u)\|^2_{\partial
T}\right)^{\frac{1}{2}}\cdot\left(\sum_{T\in\mathcal{T}_h}h_T^{-3}\|Q_be_0-e_b\|^2_{\partial
T}\right)^{\frac{1}{2}}
\\
\nonumber&& +\left(\sum_{T\in\mathcal{T}_h}\|\nabla( \Delta
u)-Q_b(\nabla(\Delta u))\|^2_{\partial
T}\right)^{\frac{1}{2}}\cdot\left(\sum_{T\in\mathcal{T}_h}\|e_0-Q_be_0\|^2_{\partial
T}\right)^{\frac{1}{2}}
\\
\nonumber&\leq& Ch^{k-1}\ \|u\|_{k+2} \3bar e_h\3bar,
\end{eqnarray}
where the $H^{k+2}$-norm of $u$ is used because the estimate in
Lemma \ref{Lemma:Lemma6.5} is not optimal in terms of the mesh
parameter $h$.

The third and fourth terms can be estimated by using the
Cauchy-Schwarz inequality and the estimates (\ref{ineqn5}) and
(\ref{ineqn6}) in Lemma \ref{Lemma:Lemma5.2} as follows
\begin{eqnarray}\label{error_estimate4}
\left|\sum_{T\in\mathcal{T}_h}h^{-1}_T\langle  \nabla Q_0u\cdot
\mathbf{n}_e-Q_b (\nabla u\cdot \mathbf{n}_e), \nabla e_0 \cdot
\mathbf{n}_e-e_n\rangle_{\partial T}\right|\leq Ch^{k-1}\|u\|_{k+1}
|\!|\!| e_h |\!|\!|
\end{eqnarray}
and
\begin{eqnarray}\label{error_estimate5}
&&\left|\sum_{T\in\mathcal{T}_h}h^{-3}_T\langle  Q_b Q_0u-Q_bu, Q_b
e_0-e_b\rangle_{\partial T}\right| \leq Ch^{k-1}\|u\|_{k+1} |\!|\!|
e_h |\!|\!|.
\end{eqnarray}
Substituting (\ref{error_estimate2})-(\ref{error_estimate5}) into
(\ref{error_estimate1}) gives
$$ |\!|\!|  e_h |\!|\!| ^2\leq Ch^{k-1}\ \|u\|_{k+2} |\!|\!| e_h |\!|\!| ,$$
which implies (\ref{error_estimate0}) and hence completes the proof.
\end{proof}

\section{Error Estimates in $L^2$}\label{Section:L2ErrorEstimates}

In this section, we shall establish some error estimates for all
three components of the error function $e_h$ in the standard $L^2$
norm.

First of all, let us derive an error estimate for the first
component of the error function $e_h$ by applying the usual duality
argument in the finite element analysis. To this end, we consider
the problem of seeking $\varphi$ such that
\begin{eqnarray}\label{dual_biharmonic}
\Delta^2 \varphi&=&e_0, \quad {\rm in} \ \Omega,
\\
\nonumber\varphi&=&0,\quad\  {\rm on}\
\partial\Omega,
\\
\nonumber\frac{\partial \varphi}{\partial \textbf{n}}&=&0,\quad\
{\rm on}\
\partial\Omega.
\end{eqnarray}

Assume that the dual problem has the $H^4$ regularity property in
the sense that the solution function $\varphi\in H^4$ and there
exists a constant $C$ such that
\begin{eqnarray}\label{regularity_property}
\|\varphi\|_4\le C\|e_0\|.
\end{eqnarray}

\begin{theorem}\label{theorem:theorem7.5}Let
$u_h\in V_h$ be the weak Galerkin finite element solution arising
from (\ref{WGalerkin_Algorithm}) with finite element functions of
order $k\geq2$. Let $k_0=\min\{3, k\}$. Assume that the exact
solution of (\ref{biharmonic})-(\ref{boundary_normal}) is
sufficiently regular such that $u\in H^{k+2}(\Omega)$ and the dual
problem (\ref{dual_biharmonic}) has the $H^4$ regularity. Then,
there exists a constant $C$ such that
\begin{eqnarray}\label{error_estimateL^20}
 \| u_0-Q_0 u \| \leq Ch^{k+k_0-2}\|u\|_{k+1},
\end{eqnarray}
which means we have a sub-optimal order of convergence for $k=2$ and
optimal order of convergence for $k\ge 3$.
\end{theorem}

\begin{proof}
Testing (\ref{dual_biharmonic}) by error function $e_0$ and then
using the integration by parts gives
\begin{eqnarray*}
\|e_0\|^2&=&(\Delta^2 \varphi, e_0)
\\
&=& \sum_{T\in \mathcal{T}_h} (\Delta \varphi, \Delta
e_0)_T+\sum_{T\in \mathcal{T}_h} \langle\nabla(\Delta \varphi)\cdot
\textbf{n}, e_0\rangle_{\partial T}-\sum_{T\in \mathcal{T}_h}
\langle\Delta \varphi, \nabla e_0\cdot \textbf{n}\rangle_{\partial
T}
\\
&=& \sum_{T\in \mathcal{T}_h} (\Delta \varphi, \Delta
e_0)_T+\sum_{T\in \mathcal{T}_h} \langle\nabla(\Delta \varphi)\cdot
\textbf{n}, e_0-e_b\rangle_{\partial T} \\
&&-\sum_{T\in
\mathcal{T}_h} \langle\Delta \varphi, (\nabla
e_0-e_n\textbf{n}_e)\cdot \textbf{n}\rangle_{\partial T},
\end{eqnarray*}
where we have used the fact that $e_n$ and $e_b$ vanishes on the
boundary of the domain $\Omega$. By letting $u=\varphi$ and
$v_0=e_h$ in (\ref{EQN:7.2}), we can rewrite the above equation as
follows
\begin{eqnarray*}
\|e_0\|^2&=&(\Delta_w Q_h\varphi,\Delta_w e_h)_h+\sum_{T\in
\mathcal{T}_h} \langle (\nabla(\Delta \varphi)-\nabla
(\mathbb{Q}_h\Delta \varphi)\cdot \textbf{n},
e_0-e_b\rangle_{\partial T}
\\
&&-\sum_{T\in \mathcal{T}_h} \langle \Delta
\varphi-\mathbb{Q}_h\Delta \varphi, (\nabla
e_0-e_n\textbf{n}_e)\cdot \textbf{n}\rangle_{\partial T}.
\end{eqnarray*}
Next, by letting $v=Q_h \varphi$, from the error equation
(\ref{error equation}), we have
\begin{eqnarray*}
(\Delta_w Q_h\varphi,\Delta_w e_h)_h&=&\sum_{T\in \mathcal{T}_h}
\langle (\Delta u- \mathbb{Q}_h \Delta u, (\nabla Q_0\varphi)\cdot
\textbf{n}-Q_b (\nabla \varphi\cdot
\textbf{n}_e)\textbf{n}_e\cdot\textbf{n}\rangle_{\partial T}
\\
&&-\sum_{T\in \mathcal{T}_h} \langle\nabla(\Delta
u-\mathbb{Q}_h\Delta u)\cdot \textbf{n},
Q_0\varphi-Q_b\varphi\rangle_{\partial T}
\\
&&-s(e_h, Q_h \varphi)+s(Q_h u, Q_h \varphi).
\end{eqnarray*}
Combining the two equations above gives
\begin{eqnarray}\label{EQN:8.6}
\|e_0\|^2&=&\sum_{T\in \mathcal{T}_h} \langle (\nabla(\Delta
\varphi)-\nabla (\mathbb{Q}_h\Delta \varphi)\cdot \textbf{n},
e_0-e_b\rangle_{\partial T}
\\
\nonumber&&-\sum_{T\in \mathcal{T}_h} \langle \Delta
\varphi-\mathbb{Q}_h\Delta \varphi, (\nabla e_0\cdot\textbf{n}_e
-e_n)\cdot \textbf{n}\rangle_{\partial T}
\\
\nonumber&&+\sum_{T\in \mathcal{T}_h} \langle (\Delta u-
\mathbb{Q}_h \Delta u, (\nabla Q_0\varphi)\cdot \textbf{n}-Q_b
(\nabla \varphi\cdot \textbf{n}_e)\textbf{n}_e\cdot
\textbf{n}\rangle_{\partial T}
\\
\nonumber&&-\sum_{T\in \mathcal{T}_h} \langle\nabla(\Delta
u-\mathbb{Q}_h\Delta u)\cdot \textbf{n},
Q_0\varphi-Q_b\varphi\rangle_{\partial T}
\\
\nonumber&&-s(e_h, Q_h \varphi)+s(Q_h u, Q_h \varphi).
\end{eqnarray}

From the Cauchy-Schwarz inequality and Lemma \ref{Lemma:Lemma5.2},
we can estimate the six terms on the right-hand side of the identity
above as follows.

For the first term, it follows from Lemma \ref{Lemma:Lemma5.2},
\ref{ineqn7} and the fact $k_0=\min\{k, 3\}\le 3$ that
\begin{eqnarray}\label{EQN:8.7}
&&\left|\sum_{T\in \mathcal{T}_h} \langle (\nabla(\Delta
\varphi)-\nabla (\mathbb{Q}_h\Delta \varphi))\cdot \textbf{n},
e_0-e_b\rangle_{\partial T}\right|
\\
\nonumber&\le&\left(\sum_{T\in \mathcal{T}_h}h_T^3\|\nabla(\Delta
\varphi)-\nabla (\mathbb{Q}_h\Delta \varphi)\|_{\partial
T}^2\right)^{\frac12} \left(\sum_{T\in \mathcal{T}_h}h_T^{-3}\|Q_b
e_0-e_b\|_{\partial T}^2\right)^{\frac12}
\\
\nonumber&&+\left(\sum_{T\in \mathcal{T}_h}\|\nabla(\Delta
\varphi)-Q_b \nabla (\Delta \varphi)\|_{\partial
T}^2\right)^{\frac12} \left(\sum_{T\in \mathcal{T}_h}\|e_0-Q_b
e_0\|_{\partial T}^2\right)^{\frac12}
\\
\nonumber&\le&\left(\sum_{T\in \mathcal{T}_h}h_T^3\|\nabla(\Delta
\varphi)-\nabla (\mathbb{Q}_h\Delta \varphi)\|_{\partial
T}^2\right)^{\frac12} |\!|\!| e_h |\!|\!|
\\
\nonumber&&+C\lambda\left(\sum_{T\in \mathcal{T}_h}\|\nabla(\Delta
\varphi)-Q_b \nabla (\Delta \varphi)\|_{\partial
T}^2\right)^{\frac12}\cdot h^{-\frac12}\|e_0\|
\\
\nonumber&&+C\left(\sum_{T\in \mathcal{T}_h}\|\nabla(\Delta
\varphi)-Q_b \nabla (\Delta \varphi)\|_{\partial
T}^2\right)^{\frac12} \cdot h^{\frac32}|\!|\!| e_h |\!|\!|
\\
\nonumber&\leq&
Ch^{k_0-1}(\|\varphi\|_{k_0+1}+h\delta_{k_0,2}\|\varphi\|_4) |\!|\!|
e_h |\!|\!| + C\lambda h^{\frac12}\|\varphi\|_4\cdot
h^{-\frac12}\|e_0\|
\\
\nonumber&&
+Ch^{k_0-\frac52}(\|\varphi\|_{k_0+1}+h\delta_{k_0,2}\|\varphi\|_4)
\cdot h^{\frac32}|\!|\!| e_h |\!|\!|
\\
\nonumber&\leq& C h^{k_0-1}\|\varphi\|_4 |\!|\!| e_h|\!|\!| + C
\lambda \|\varphi\|_4 \|e_0\|.
\end{eqnarray}
For the second term, it follows from (\ref{ineqn1}) with $m=k_0$
that
\begin{eqnarray}\label{EQN:8.8}
&&\left|\sum_{T\in \mathcal{T}_h} \langle \Delta
\varphi-\mathbb{Q}_h\Delta \varphi, (\nabla e_0\cdot
\textbf{n}_e-e_n)\cdot \textbf{n}\rangle_{\partial T}\right|
\\
\nonumber&\le &\left(\sum_{T\in \mathcal{T}_h}h_T\|\Delta
\varphi-\mathbb{Q}_h\Delta \varphi\|^2_{\partial T}
\right)^{\frac12}\left(\sum_{T\in \mathcal{T}_h}h_T^{-1}\|\nabla
e_0\cdot \textbf{n}_e-e_n\|^2_{\partial T}\right)^{\frac12}
\\
\nonumber&\leq& Ch^{k_0-1}\|\varphi\|_{k_0+1} |\!|\!| e_h |\!|\!|
\le C h^{k_0-1}\|\varphi\|_4 |\!|\!| e_h |\!|\!|.
\end{eqnarray}
As to the third term, it follows from Cauchy-Schwarz inequality and
Lemma \ref{Lemma:Lemma5.2} that
\begin{eqnarray}\label{EQN:8.9}
&&\left|\sum_{T\in \mathcal{T}_h} \langle \Delta u- \mathbb{Q}_h
\Delta u, (\nabla Q_0\varphi)\cdot \textbf{n}-Q_b (\nabla
\varphi\cdot \textbf{n}_e)\textbf{n}_e\cdot
\textbf{n}\rangle_{\partial T}\right|
\\
\nonumber&\le &\left(\sum_{T\in \mathcal{T}_h}h_T\|\Delta u-
\mathbb{Q}_h \Delta u\|^2_{\partial T}
\right)^{\frac12}\left(\sum_{T\in \mathcal{T}_h}h_T^{-1}\|(\nabla
Q_0\varphi)\cdot \textbf{n}-Q_b (\nabla \varphi\cdot
\textbf{n}_e)\|^2_{\partial T}\right)^{\frac12}
\\
\nonumber&\leq& Ch^{k-1}\|u\|_{k+1}h^{k_0-1}\|\varphi\|_{k_0+1} \le
Ch^{k+k_0-2}\|u\|_{k+1}\|\varphi\|_{4} .
\end{eqnarray}
For the forth term, by using Lemma \ref{Lemma:Lemma5.1}, we have
\begin{eqnarray}\label{EQN:8.10}
&&\left|\sum_{T\in \mathcal{T}_h} \langle\nabla(\Delta
u-\mathbb{Q}_h\Delta u)\cdot \textbf{n},
Q_0\varphi-Q_b\varphi\rangle_{\partial T}\right|
\\
\nonumber&\le &\left(\sum_{T\in \mathcal{T}_h}h_T^3\|\nabla(\Delta
u-\mathbb{Q}_h\Delta u)\|^2_{\partial T}
\right)^{\frac12}\left(\sum_{T\in
\mathcal{T}_h}h_T^{-3}\|Q_0\varphi-\varphi\|^2_{\partial
T}\right)^{\frac12}
\\
\nonumber&\le &
Ch^{k-1}(\|u\|_{k+1}+h\delta_{k,2}\|u\|_4)h^{t_0-1}\|\varphi\|_{k_0+1}
\\
\nonumber&\le &
Ch^{k-1}(\|u\|_{k+k_0-2}+h\delta_{k,2}\|u\|_4)\|\varphi\|_{4}.
\end{eqnarray}
As to the fifth term, we also use the Cauchy-Schwarz inequality and
Lemma \ref{Lemma:Lemma5.2} to obtain
\begin{eqnarray}\label{EQN:8.11}
&&|s(e_h, Q_h \varphi)|
\\
\nonumber&\le & \left|\sum_{T\in \mathcal{T}_h}h_T^{-1} \langle
\nabla e_0\cdot\textbf{n}_e-e_n, \nabla Q_0\varphi\cdot
\textbf{n}_e-Q_b (\nabla \varphi\cdot \textbf{n}_e)\rangle_{\partial
T}\right|
\\
\nonumber &&+ \left|\sum_{T\in \mathcal{T}_h}h_T^{-3} \langle Q_b
e_0-e_b, Q_b Q_0\varphi - Q_b\varphi \rangle_{\partial T}\right|
\\
\nonumber&\le & Ch^{k_0-1}\|\varphi\|_4 |\!|\!| e_h |\!|\!|.
\end{eqnarray}
The last term can be estimated as follows
\begin{eqnarray}\label{EQN:8.12}
&&|s(Q_h u, Q_h \varphi)|
\\
\nonumber&\le & \left|\sum_{T\in \mathcal{T}_h}h_T^{-1}
\langle(\nabla Q_0u\cdot \textbf{n}_e-Q_b (\nabla u\cdot
\textbf{n}_e), (\nabla Q_0\varphi\cdot \textbf{n}_e-Q_b (\nabla
\varphi\cdot \textbf{n}_e)\rangle_{\partial T}\right|
\\
\nonumber&&+ \left|\sum_{T\in \mathcal{T}_h}h_T^{-3} \langle Q_b Q_0
u- Q_b u, Q_b Q_0\varphi - Q_b\varphi \rangle_{\partial T}\right|
\\
\nonumber&\le & Ch^{k-1}\|u\|_{k+1} h^{k_0-1} \|\varphi\|_{k_0+1}
\\
\nonumber&\le & Ch^{k+k_0-2}\|u\|_{k+1} \|\varphi\|_{4}.
\end{eqnarray}
Substituting all the six estimates into (\ref{EQN:8.6}) we obtain
\begin{eqnarray*}
\|e_0\|^2\le
&&Ch^{k+k_0-2}(\|u\|_{k+1}+h\delta_{k,2}\|u\|_4)\|\varphi\|_{4} \\
&&+C h^{k_0-1}\|\varphi\|_4 |\!|\!| e_h |\!|\!|+C \lambda
\|\varphi\|_4 \|e_0\|.
\end{eqnarray*}
Using the regularity estimate (\ref{regularity_property}) and
choosing constant $\lambda$ such that $C \lambda
\|\varphi\|_4<\frac12\|e_0\|$, we arrive at
\begin{eqnarray*}
\|e_0\|&\le& C h^{k_0-1} |\!|\!| e_h |\!|\!|
+Ch^{k+k_0-2}(\|u\|_{k+1}+h\delta_{k,2}\|u\|_4)
\\
&\le& Ch^{k+k_0-2}\|u\|_{k+2}.
\end{eqnarray*}
Together with the $H^2$ error estimate (\ref{error_estimate0}) we
have the desired $L^2$ error estimate (\ref{error_estimateL^20}).
\end{proof}

In order to study the error estimates on edges, we shall introduce
the edge-based $L^2$ norm here. To keep the consistency of order,
the edge-based $L^2$ norm is different from the standard $L^2$ norm.

\begin{definition}\label{edge_norm}
For any function $v$ defined on the edges $\mathcal{E}_h$,
\begin{eqnarray}
\nonumber\|v\|^2_{\mathcal{E}_h}&=&\sum_{e\in\hspace{0.15em}\mathcal{E}_h}h_e
\|v\|^2_{L^2(e)},
\end{eqnarray}
where $h_e$ is the measure of edge $e\in \mathcal{E}_h$.
\end{definition}

Next, we shall derive the estimates for the second and third
components of the error function $e_h$.

\begin{theorem}\label{theorem:theorem7.6}Let
$u_h\in V_h$ be the weak Galerkin finite element solution arising
from (\ref{WGalerkin_Algorithm}) with finite element functions of
order $k\geq2$. Let $k_0=\min\{k,3\}$. Assume that the exact
solution of (\ref{biharmonic})-(\ref{boundary_normal}) is
sufficiently regular such that $u\in H^{k+2}(\omega)$ and the dual
problem (\ref{dual_biharmonic}) has the $H^4$ regularity property.
Then, there exists a constant $C$ such that
\begin{eqnarray}\label{error_estimateL^20-new}
 \| u_b-Q_b u \|_{\mathcal{E}_h}&\leq& Ch^{k+k_0-2}\|u\|_{k+2},
\\ \label{error_estimateL^20-new2}  \| u_n-Q_b
(\nabla u_0\cdot{\bf n}_e)\|_{\mathcal{E}_h}&\leq&
 Ch^{k+k_0-3}\|u\|_{k+2}.
\end{eqnarray}
\end{theorem}

\begin{proof}
It is obvious that
\begin{eqnarray*}
\|e_b\|^2_{L^2(e)}\leq 2(\|Q_be_0\|^2_{L^2(e)}+\|Q_be_0-e_b\|^2_{L^2(e)}).
\end{eqnarray*}
Summing over all edges, we have
\begin{eqnarray}\label{EQN:8.13}
\quad\| u_b-Q_b u \|^2_{\mathcal{E}_h}&=&\sum_{e\in\mathcal{E}_h}h_e
\|u_b-Q_b u\|^2_{L^2(e)}\\
\nonumber&\leq&2\left(\sum_{e\in\mathcal{E}_h}h_e\|Q_be_0\|^2_{L^2(e)}
+\sum_{e\in\mathcal{E}_h}h_e\|Q_be_0-e_b\|^2_{L^2(e)}\right)\\
\nonumber&&\leq
C\left(\sum_{T\in\mathcal{T}_h}h_T\|Q_be_0\|^2_{L^2(\partial T)}
+\sum_{T\in\mathcal{T}_h}h_T\|Q_be_0-e_b\|^2_{L^2(\partial
T)}\right).
\end{eqnarray}

We shall discuss the two terms separately. For the first part, by
applying the trace inequality (\ref{Trace inequality00}), the
inverse inequality (\ref{Inverse Inequality00}) and the error
estimate for $e_0$ in Theorem \ref{theorem:theorem7.5}, we have
\begin{eqnarray}
\sum_{T\in\mathcal{T}_h}h_T\|Q_be_0\|^2_{L^2(\partial T)}&\leq&\sum_{T\in\mathcal{T}_h}h_T\|e_0\|^2_{L^2(\partial T)}\\
\nonumber&\leq& C\sum_{T\in\mathcal{T}_h}(\|e_0\|^2_{L^2(T)}+h_T^2
\|\nabla e_0\|^2_{L^2(T)})\\
\nonumber&\leq& C\sum_{T\in\mathcal{T}_h}\|e_0\|^2_{L^2(T)}\\
\nonumber&\leq& Ch^{2k+2k_0-4}\|u\|_{k+2}^2.
\end{eqnarray}

For the second part, we use the trip-bar norm to handle the second
part.
\begin{eqnarray}
\sum_{T\in\mathcal{T}_h}h_T\|Q_be_0-e_b\|^2_{L^2(\partial T)}
&\leq& h^4\sum_{T\in\mathcal{T}_h}h_T^{-3}\|Q_be_0-e_b\|^2_{L^2(\partial T)}\leq h^4|\!|\!|e_h|\!|\!|^2\\
\nonumber&\leq& Ch^{2k+2k_0-4}\|u\|_{k+2}^2.
\end{eqnarray}
Combining the above two estimates gives the desired error estimate
(\ref{error_estimateL^20-new}).

Similarly, we establish the error estimates for $e_n$.
\begin{eqnarray}
\|e_n\|^2_{\mathcal{E}_h}&=&\sum_{e\in\mathcal{E}_h}h_e\|e_n\|^2_{L^2(e)}\\
\nonumber&\leq& C\left(\sum_{T\in\mathcal{T}_h}h_T\|\nabla e_0\cdot
{\bf n}_e\|_{\partial T}+\sum_{T\in\mathcal{T}_h}h_T\|\nabla
e_0\cdot
{\bf n}_e-e_n\|_{\partial T}\right)\\
\nonumber&\leq& C\left(\sum_{T\in\mathcal{T}_h}h_T\|\nabla
e_0\|_{\partial T}+ h^2\sum_{T\in\mathcal{T}_h}h_T^{-1}\|\nabla
e_0\cdot
{\bf n}_e-e_n\|_{\partial T}\right)\\
\nonumber&\leq& C\left(\sum_{T\in\mathcal{T}_h}\|\nabla e_0\|_{T}+
h^2|\!|\!|e_h|\!|\!|\right)\\
\nonumber&\leq& C\left(\sum_{T\in\mathcal{T}_h}h_T^{-2}\|e_0\|_{T}+
h^2|\!|\!|e_h|\!|\!|\right)\\
\nonumber&\leq& C(h^{2k+2k_0-6}+h^{2k})\|u\|_{k+2}^2.
\end{eqnarray}

Thus, we have
\begin{eqnarray*}
\|e_n\|_{\mathcal{E}_h}\leq Ch^{k+k_0-3}\|u\|_{k+2},
\end{eqnarray*}
which completes the proof.
\end{proof}

\section{Numerical Results}\label{Section:NumericalResults}

In this section, we would like to report some numerical results for
the weak Galerkin finite element method proposed and analyzed in
previous sections. Here we use the following finite element space
$$
\tilde{V}_h=\{v=\{v_0, v_b, v_n \textbf{n}_e \}, v_0\in P_2(T), v_b,
v_n\in P_{1}(e), T\in \mathcal{T}_h, e\subset \mathcal{E}_h\}.
$$
For any given $v=\{v_0, v_b, v_n \textbf{n}_e \}\in \tilde{V}_h$ and
$\varphi\in P_0(T)$, we compute the discrete weak Laplacian
$\Delta_w v$ on each element $T$ as a function in $P_0(T)$ as
follows
\begin{eqnarray*}
\qquad(\Delta_{w} v, \varphi)_T=(v_0, \Delta \varphi)_T-\langle v_b,
\nabla \varphi\cdot \textbf{n}\rangle_{\partial T}+\langle
v_n\textbf{n}_e\cdot \textbf{n},\varphi\rangle_{\partial T},
\end{eqnarray*}
which could be simplified as
\begin{eqnarray*}
\qquad(\Delta_{w} v, \varphi)_T=\langle  v_n\textbf{n}_e\cdot
\textbf{n},\varphi\rangle_{\partial T}.
\end{eqnarray*}

The error for the weak Galerkin solution is measured in six norms
defined as follows:
\begin{eqnarray*}
\begin{array}{c} \displaystyle|\!|\!| e_h |\!|\!|^2=\sum_{T\in\mathcal{T}_h}
\left(\int_T|\Delta_w v_h|^2 dT+ h_T^{-1}\int_{\partial T}|(\nabla
v_0)\cdot\textbf{n}_e-v_n |^2 ds\right.
\\[2mm]
\displaystyle\hspace{2cm}\left.+ h_T^{-3}\int_{\partial
T}(Q_bv_0-v_b)^2 ds\right) \qquad {\rm (A \ discrete } \ H^2\ {\rm
norm)}\nonumber
\\[3mm]
\displaystyle\|Q_0 v-v_0\|^2=\sum_{T\in\mathcal{T}_h} \int_T|Q_0 v-
v_0|^2 dT \qquad {\rm (Element \ based} \ L^2\ {\rm norm)}
\\[3mm]
\displaystyle\|Q_b
v-v_b\|^2_{\mathcal{E}_h}=\sum_{e\in\mathcal{E}_h} h_e\int_e|Q_b v-
v_b|^2 ds \qquad {\rm (Edge \ based} \ L^2\ {\rm norm\ for} \ v_b
{\rm )}
\\[3mm]
\displaystyle\|Q_b
v-v_n\|^2_{\mathcal{E}_h}=\sum_{e\in\mathcal{E}_h} h_e\int_e|Q_b  v-
v_n|^2 ds \qquad {\rm (Edge \ based} \ L^2\ {\rm norm\ for} \ v_n
{\rm )}
\\[3mm]
\displaystyle\|Q_b v-v_b\|_{\infty}=\max_{e\in\mathcal{E}_h}\{|Q_b
v- v_b|\} \qquad {\rm (Edge \ based} \ L^\infty\ {\rm norm\ for} \
v_b {\rm )}
\\[3mm]
\displaystyle\|Q_b  v-v_n\|_{\infty}=\max_{e\in\mathcal{E}_h}\{|Q_b
(\nabla u_0\cdot{\bf n}_e)- v_n|\} \qquad {\rm (Edge \ based} \
L^\infty\ {\rm norm\ for} \ v_n {\rm )}
\end{array}
\end{eqnarray*}

\begin{example}\label{Example1}{\rm Consider the biharmonic problem (\ref{biharmonic})-(\ref{boundary_normal}) in the square domain
$\Omega=(0,1)^2$. It has the analytic solution
$u(x)=x^2(1-x)^2y^2(1-y)^2$, and the right hand side function $f$ in
(\ref{biharmonic}) is computed to match the exact solution. The mesh
size is denoted by $h=1/n$. Table 6.1 shows that the convergence
rates for the WG-FEM solution in the $H^2$ and $L^2$ norms are of
order $O(h)$ and $O(h^2)$ when $k=2$, respectively.

Table 6.2 shows that the errors and orders of Example \ref{Example1}
in $L^2$ and $L^\infty$ for $e_b$. The numerical results are in
consistency with theory for these two cases.

Table 6.3 shows that the errors and orders of Example \ref{Example1}
in $L^2$ and $L^\infty$ for $e_n$. The numerical results are in
consistency with theory for these two cases. }

\begin{center}
Table 6.1. Errors and orders of Example \ref{Example1} in $H^2$ and $L^2$
with $k=2$. \\
\begin{center}
    \begin{tabular}{ | c || c | c | c | c |}
    \hline
     $h$  & \  $|\!|\!| u_h -Q_h u|\!|\!|$\  & \quad \  order \qquad \,  &  \  $\|u_0-Q_0 u\|$  \   & \quad order \quad \\ \hline \hline
     3.74355e-01   &  3.69061e-01 &  & 4.29897e-02
 &   \\ \hline
    1.91955e-01
 &   1.89785e-01
 &   9.59493e-01
 &   1.11418e-02
 &   1.94801
  \\ \hline
    9.56362e-02
 & 1.01110e-01
 & 9.08440e-01
 & 2.97175e-03
 & 1.90660
  \\ \hline
4.78382e-02
 & 5.57946e-02
 & 8.57728e-01
 & 8.08649e-04
 & 1.87773
  \\ \hline
 2.20971e-02
 & 3.00721e-02
 & 8.91700e-01
 & 2.14457e-04
 & 1.91483
  \\ \hline
 1.10485e-02
 & 1.55286e-02
 & 9.53498e-01
 &5.49264e-05
 & 1.96512
  \\ \hline
    \hline
    \end{tabular}
\end{center}
\end{center}
\newpage

\begin{center}
Table 6.2. Errors and orders of Example \ref{Example1} in $L^2$ and $L^\infty$ for $e_b$
with $k=2$. \\
\begin{center}
    \begin{tabular}{ | c || c | c | c | c |}
    \hline
     $h$  & \  $\|Q_b u-u_b\|_{\mathcal{E}_h}$\  & \quad \  order \qquad \,  &  \  $\|Q_b u-u_b\|_{\infty}$  \   & \quad order \quad \\ \hline \hline
     3.74355e-01   &  1.21967e-01
 &  & 1.18101e-01
 &   \\ \hline
    1.91955e-01  &   3.12884e-02
 &   1.91858
 &   3.27686e-02
 &   1.84964
  \\ \hline
    9.56362e-02
 & 8.39049e-03
 & 1.89880
 & 8.84728e-03
 & 1.88901
  \\ \hline
4.78382e-02
 & 2.28623e-03
 & 1.87578
 & 2.39957e-03
 & 1.88246
  \\ \hline
 2.20971e-02
 & 6.06514e-04
 & 1.91436
 & 6.33868e-04
 & 1.92052
  \\ \hline
 1.10485e-02
 & 1.55351e-04
 & 1.96501
 & 1.62044e-04
 & 1.96780
  \\ \hline
    \hline
    \end{tabular}
\end{center}
\end{center}

\begin{center}
Table 6.3. Errors and orders of Example \ref{Example1} in $L^2$ and $L^\infty$ for $e_n$
with $k=2$. \\
\begin{center}
    \begin{tabular}{ | c || c | c | c | c |}
    \hline
     $h$  & \  $\|Q_b (\nabla u\cdot{\bf n}_e)-u_n\|_{\mathcal{E}_h}$\  & \  order  \,  &  \  $\|Q_b (\nabla u\cdot{\bf n}_e)-u_n\|_{\infty}$  \   &  order  \\ \hline \hline
     3.74355e-01   &  1.18286e-01
     &  & 5.28497e-02
     &   \\ \hline
     1.91858e+00   &   3.12884e-02
 &   1.91858
 &   1.51029e-02
 &   1.80707
  \\ \hline
    9.56362e-02
 & 8.39049e-03
 & 1.89880
 & 7.33970e-03
 & 1.04103
  \\ \hline
4.78382e-02
 & 2.28623e-03
 & 1.87578
 & 3.41617e-03
 & 1.10334
  \\ \hline
 2.20971e-02
 & 6.06514e-04
 & 1.91436
 & 1.18287e-03
 & 1.53009
  \\ \hline
 1.10485e-02
 & 1.55351e-04
 & 1.96501
 & 3.30602e-04
 & 1.83912
  \\ \hline
    \hline
    \end{tabular}
\end{center}
\end{center}

{\color{blue} In Tables 6.4-6.6 we investigate the same problem for
$k = 3$.  Table 6.4 shows that the convergence rates for the WG-FEM
solution in the $H^2$ and $L^2$ norms are of order $O(h^2)$ and
$O(h^4)$. Table 6.5 and 6.6 show the errors and orders in $L^2$ and
$L^\infty$ for $e_b$ and $e_n$, which are also consistent with
theoretical conclusions.}

{\color{blue}
\begin{center}
Table 6.4. Errors and orders of example \ref{Example1} in $H^2$ and
$L^2$
with $k=3$. \\
\begin{center}
    \begin{tabular}{ | c || c | c | c | c |}
    \hline
     $h$  & \  $|\!|\!| u_h -Q_h u|\!|\!|$\  & \quad \  order \qquad \,  &  \  $\|u_0-Q_0 u\|$  \   & \quad order \quad \\
     \hline \hline
     3.74355e-01 & 1.17819e-01 &         & 4.56114e-03 &
     \\ \hline
     1.91955e-01 & 3.56257e-02 & 1.72558 & 4.16403e-04 & 3.45334
     \\ \hline
     9.56362e-02 & 1.00915e-02 & 1.81977 & 3.55158e-05 & 3.55145
     \\ \hline
     4.78382e-02 & 2.56977e-03 & 1.97343 & 2.30985e-06 & 3.94259
     \\ \hline
     2.20971e-02 & 6.44317e-04 & 1.99580 & 1.44990e-07 & 3.99378
     \\ \hline
     1.10485e-02 & 1.61222e-04 & 1.99873 & 9.07702e-09 & 3.99759
     \\ \hline
    \hline
    \end{tabular}
\end{center}
\end{center}
\newpage

\begin{center}
Table 6.5. Errors and orders of example \ref{Example1} in $L^2$ and
$L^\infty$ for $e_b$
with $k=3$. \\
\begin{center}
    \begin{tabular}{ | c || c | c | c | c |}
    \hline
     $h$  & \  $\|Q_b u-u_b\|_{\mathcal{E}_h}$\  & \quad \  order \qquad \,  &  \  $\|Q_b u-u_b\|_{\infty}$  \   & \quad order \quad \\
     \hline \hline
     3.74355e-01 & 8.34847e-03 &         & 1.15414e-02 &
     \\ \hline
     1.91955e-01 & 8.06272e-04 & 3.37217 & 1.08014e-03 & 3.41753
     \\ \hline
     9.56362e-02 & 7.89345e-05 & 3.35254 & 9.02080e-05 & 3.58181
     \\ \hline
     4.78382e-02 & 5.19889e-06 & 3.92438 & 5.93961e-06 & 3.92481
     \\ \hline
     2.20971e-02 & 3.26604e-07 & 3.99259 & 3.72799e-07 & 3.99390
     \\ \hline
     1.10485e-02 & 2.04554e-08 & 3.99699 & 2.33003e-08 & 3.99998
     \\ \hline
    \hline
    \end{tabular}
\end{center}
\end{center}

\begin{center}
Table 6.6. Errors and orders of example \ref{Example1} in $L^2$ and
$L^\infty$ for $e_n$
with $k=3$. \\
\begin{center}
    \begin{tabular}{ | c || c | c | c | c |}
    \hline
     $h$  & \  $\|Q_b  (\nabla u\cdot{\bf n}_e)-u_n\|_{\mathcal{E}_h}$\  & \  order  \,  &  \  $\|Q_b  (\nabla u\cdot{\bf n}_e)-u_n\|_{\infty}$  \   &  order  \\
     \hline \hline
     3.74355e-01   &  5.23031e-02
     &  & 1.15371e-01
     &
     \\ \hline
     1.91858e+00 & 8.83906e-03 & 2.56493 & 1.96390e-02 & 2.55449
     \\ \hline
     9.56362e-02 & 1.50030e-03 & 2.55865 & 3.59916e-03 & 2.44799
     \\ \hline
     4.78382e-02 & 1.89000e-04 & 2.98878 & 4.60320e-04 & 2.96695
     \\ \hline
     2.20971e-02 & 2.33468e-05 & 3.01709 & 5.56932e-05 & 3.04707
     \\ \hline
     1.10485e-02 & 2.89988e-06 & 3.00916 & 6.86324e-06 & 3.02054
     \\ \hline
    \hline
    \end{tabular}
\end{center}
\end{center}}

\end{example}
\begin{example}\label{Example2}
{\rm Consider the biharmonic problem
(\ref{biharmonic})-(\ref{boundary_normal}) in the square domain
$\Omega=(0,1)^2$. It has the analytic solution $u(x)=\sin (\pi
x)\sin (\pi y)$, and the right hand side function $f$ in
(\ref{biharmonic}) is computed accordingly.

The numerical results are presented in Tables 6.7-6.12 which confirm
the theory developed in previous sections. }

\begin{center}
Table 6.7. Errors and orders of Example \ref{Example2} in $H^2$ and $L^2$
with $k=2$. \\
\begin{center}
    \begin{tabular}{ | c || c | c | c | c |}
    \hline
     $h$  & \  $|\!|\!| u_h -Q_h u|\!|\!|$\  & \quad \  order \qquad \,  &  \  $\|u_0-Q_0 u\|$  \   & \quad order \quad \\ \hline \hline
     3.74355e-01   &  3.51847e+01
 &  & 4.18608E+00
 &   \\ \hline
    1.91955e-01
 &   1.79831e+01
 &   9.68306e-01
 &   1.06553e+00
 &   1.97403
  \\ \hline
    9.56362e-02
 & 9.36621e+00
 & 9.41104e-01
 & 2.74735e-01
 & 1.95546
  \\ \hline
4.78382e-02
 & 4.90899e+00
 & 9.32039e-01
 & 7.07013e-02
 & 1.95823

  \\ \hline
 2.20971e-02
 & 2.51557e+00
 & 9.64541e-01
 & 1.79112e-02
 & 1.98087
  \\ \hline
 1.10485e-02
 & 1.26858e+00
 & 9.87671e-01
 & 4.49750e-03
 & 1.99367
  \\ \hline
    \hline
    \end{tabular}
\end{center}
\end{center}

\newpage
\begin{center}
Table 6.8. Errors and orders of Example \ref{Example2} in $L^2$ and $L^\infty$ for $e_b$
with $k=2$. \\
\begin{center}
    \begin{tabular}{ | c || c | c | c | c |}
    \hline
     $h$  & \  $\|Q_b u-u_b\|_{\mathcal{E}_h}$\  & \quad \  order \qquad \,  &  \  $\|Q_b u-u_b\|_{\infty}$  \   & \quad order \quad \\ \hline \hline
     3.74355e-01
     &  1.15398e+01
     &  & 1.10028e+01
     &
     \\ \hline
    1.91955e-01  &   2.99335e+00
 &   1.94679
 &   2.97577e+00
 &   1.88654
  \\ \hline
    9.56362e-02
 & 7.75705e-01
 & 1.94818
 & 7.77671e-01
 & 1.93603
  \\ \hline
4.78382e-02
 & 1.99884e-01
 & 1.95635
 & 2.00300e-01
 & 1.95700
  \\ \hline
 2.20971e-02
 & 5.06547e-02
 & 1.98039
 & 5.07058e-02
 & 1.98194
  \\ \hline
 1.10485e-02
 & 1.27205e-02
 & 1.99354
 & 1.27268e-02
 & 1.99428
  \\ \hline
    \hline
    \end{tabular}
\end{center}
\end{center}
\vskip1cm

\begin{center}
Table 6.9. Errors and orders of Example \ref{Example2} in $L^2$ and $L^\infty$ for $e_n$
with $k=2$. \\
\begin{center}
    \begin{tabular}{ | c || c | c | c | c |}
    \hline
    $h$  & \  $\|Q_b  (\nabla u\cdot{\bf n}_e)-u_n\|_{\mathcal{E}_h}$\  &  \  order  \,  &  \  $\|Q_b  (\nabla u\cdot{\bf n}_e)-u_n\|_{\infty}$  \   &  order  \\ \hline \hline
     3.74355e-01   &  1.15398e+01
 &  & 4.02986e+00
 &   \\ \hline
    1.91955e-01  &   2.99335e+00
 &   1.94679
 &   1.26437e+00
 &   1.67231
  \\ \hline
    9.56362e-02
 & 7.75705e-01
 & 1.94818
 & 4.40635e-01
 & 1.52076
  \\ \hline
4.78382e-02
 & 1.99884e-01
 & 1.95635
 & 1.74400e-01
 & 1.33718
  \\ \hline
 2.20971e-02
 & 5.06547e-02
 & 1.98039
 & 5.22660E-02
 & 1.73846
  \\ \hline
 1.10485e-02
 & 1.27205e-02
 & 1.99354
 & 1.37655e-02
 & 1.92482
  \\ \hline
    \hline
    \end{tabular}
\end{center}
\end{center}

{\color{blue}
\begin{center}
Table 6.10. Errors and orders of example \ref{Example2} in $H^2$ and
$L^2$
with $k=3$. \\
\begin{center}
    \begin{tabular}{ | c || c | c | c | c |}
    \hline
     $h$  & \  $|\!|\!| u_h -Q_h u|\!|\!|$\  & \quad \  order \qquad \,  &  \  $\|u_0-Q_0 u\|$  \   & \quad order \quad \\
     \hline \hline
     3.74355e-01 & 9.17084e+00 &         & 3.37369e-01 &
     \\ \hline
     1.91955e-01 & 2.46720e+00 & 1.89418 & 2.77383e-02 & 3.60438
     \\ \hline
     9.56362e-02 & 6.52418e-01 & 1.91900 & 2.14578e-03 & 3.69231
     \\ \hline
     4.78382e-02 & 1.65736e-01 & 1.97691 & 1.36946e-04 & 3.96982
     \\ \hline
     2.20971e-02 & 4.16442e-02 & 1.99270 & 8.50154e-06 & 4.00974
     \\ \hline
     1.10485e-02 & 1.04302e-02 & 1.99734 & 5.29568e-07 & 4.00484
     \\ \hline
    \hline
    \end{tabular}
\end{center}
\end{center}
\newpage

\begin{center}
Table 6.11. Errors and orders of example \ref{Example2} in $L^2$ and
$L^\infty$ for $e_b$
with $k=3$. \\
\begin{center}
    \begin{tabular}{ | c || c | c | c | c |}
    \hline
     $h$  & \  $\|Q_b u-u_b\|_{\mathcal{E}_h}$\  & \quad \  order \qquad \,  &  \  $\|Q_b u-u_b\|_{\infty}$  \   & \quad order \quad \\
     \hline \hline
     3.74355e-01 & 5.34596e-01 &         & 7.40358e-01 &
     \\ \hline
     1.91955e-01 & 4.42882e-02 & 3.59346 & 6.53790e-02 & 3.50132
     \\ \hline
     9.56362e-02 & 3.79823e-03 & 3.54353 & 5.37615e-03 & 3.60418
     \\ \hline
     4.78382e-02 & 2.44771e-04 & 3.95582 & 3.54304e-04 & 3.92351
     \\ \hline
     2.20971e-02 & 1.51601e-05 & 4.01308 & 2.24297e-05 & 3.98151
     \\ \hline
     1.10485e-02 & 9.43450e-07 & 4.00619 & 1.40631e-06 & 3.99543
     \\ \hline
    \hline
    \end{tabular}
\end{center}
\end{center}

\begin{center}
Table 6.12. Errors and orders of example \ref{Example2} in $L^2$ and
$L^\infty$ for $e_n$
with $k=3$. \\
\begin{center}
    \begin{tabular}{ | c || c | c | c | c |}
    \hline
     $h$  & \  $\|Q_b  (\nabla u\cdot{\bf n}_e)-u_n\|_{\mathcal{E}_h}$\  &  \  order  \,  &  \  $\|Q_b  (\nabla u\cdot{\bf n}_e)-u_n\|_{\infty}$  \   &  order  \\
     \hline \hline
     3.74355e-01 & 3.44921e+00 &         & 8.19181e+00 &
     \\ \hline
     1.91858e+00 & 5.38035e-01 & 2.68049 & 1.42296e+00 & 2.52529
     \\ \hline
     9.56362e-02 & 7.93752e-02 & 2.76094 & 2.26764e-01 & 2.64963
     \\ \hline
     4.78382e-02 & 9.99040e-03 & 2.99007 & 3.26867e-02 & 2.79442
     \\ \hline
     2.20971e-02 & 1.23394e-03 & 3.01727 & 4.33160e-03 & 2.91573
     \\ \hline
     1.10485e-02 & 1.52755e-04 & 3.01398 & 5.50588e-04 & 2.97585
     \\ \hline
    \hline
    \end{tabular}
\end{center}
\end{center}}

\end{example}

\appendix
\section{$L^2$ Projection and Some Technical Results.}\label{Section:ApprProperties}

In this section, we shall present some technical results for the
$L^2$ projection operators with respect to the finite element space
$V_h$. These results are useful for the error estimates of the WG
finite element method.

\begin{lemma}{\rm(\cite{WY2})}\label{Trace inequality}
~\emph{\rm (}Trace Inequality{\rm)} Let $\mathcal{T}_h$ be a
partition of the domain $\Omega$ into polygons in 2D or polyhedra in
3D. Assume that the partition $\mathcal{T}_h$ satisfies the
assumptions (A1), (A2), and (A3) as specified in \cite{WY2}. Then,
there exists a constant $C$ such that for any $T\in \mathcal{T}_h$
and edge/face $e\in\partial T$, we have
\begin{eqnarray}\label{Trace inequality00}\|\theta\|^p_{e}\leq
Ch_T^{-1}(\|\theta\|^p_{T}+h^p_T\|\nabla\theta\|^p_{T}),
\end{eqnarray}
where $\theta\in H^{1}(T)$ is any function.
\end{lemma}

\begin{lemma}{\rm(\cite{WY2})}\label{Inverse Inequality}~{\rm (}Inverse Inequality{\rm )} Let
$\mathcal{T}_h$ be a partition of the domain $\Omega$ into polygons
or polyhedra. Assume that $\mathcal{T}_h$ satisfies all the
assumptions (A1)-(A4) as specified in \cite{WY2}. Then, there exists
a constant $C(n)$ such that
\begin{eqnarray}\label{Inverse Inequality00}
\|\nabla\varphi\|_{T}\leq C(n)h^{-1}_T\|\varphi\|_{T},\quad\forall
T\in\mathcal{T}_h
\end{eqnarray}
for any piecewise polynomial $\varphi$ of degree $n$ on
$\mathcal{T}_h$.
\end{lemma}

\subsection{Approximation properties}

The following lemma provides some approximation properties for the
projection operators $Q_h$ and $\mathbb{Q}_h$.

\begin{lemma}{\rm(\cite{MWY3})}\label{Lemma:Lemma5.1} Let $\mathcal{T}_h$ be a finite element partition of
$\Omega$ satisfying the shape regularity assumptions. Then, for any
$0\leq s\leq 2$ and $2\leq m\leq k$ we have
\begin{eqnarray}
\label{ineqn1}&\sum_{T\in \mathcal{T}_h}h^{2s}_T\|u-Q_0u\|^2_{s,T}\leq Ch^{2(m+1)}\|u\|^2_{m+1},\\
\label{ineqn2}&\sum_{T\in \mathcal{T}_h}h^{2s}_T\|\Delta
u-\mathbb{Q}_h\Delta u\|^2_{s,T}\leq Ch^{2(m-1)}\|u\|^2_{m+1}.
\end{eqnarray}
\end{lemma}

\begin{lemma}\label{Lemma:Lemma5.2} Let\ $ 2 \leq m\leq k, \omega \in
H^{m+2}(\Omega)$. There exists a constant $C$ such that the
following estimates hold true:
\begin{eqnarray}\label{ineqn3}
\left(\sum_{T\in\mathcal{T}_h} h_T\| \Delta
\omega-\mathbb{Q}_h\Delta\omega\|^2_{\partial
T}\right)^{\frac{1}{2}}\leq Ch^{m-1}\|\omega\|_{m+1},
\end{eqnarray}
\begin{eqnarray}\label{ineqn4}
\left(\sum_{T\in\mathcal{T}_h}h^3_T\|\nabla(\Delta\omega-\mathbb{Q}_h\Delta\omega)\|^2_{\partial
T}\right)^{\frac{1}{2}}\leq
Ch^{m-1}(\|\omega\|_{m+1}+h\delta_{m,2}\|\omega\|_4),
\end{eqnarray}
\begin{eqnarray}\label{ineqn5}
\left(\sum_{T\in\mathcal{T}_h}h^{-1}_T\|\nabla(Q_0\omega)\cdot
 \textbf{n}_e -Q_b(\nabla\omega\cdot \textbf{n}_e)\|^2_{\partial
T}\right)^{\frac{1}{2}}\leq Ch^{m-1}\|\omega\|_{m+1},
\end{eqnarray}
\begin{eqnarray}\label{ineqn6}
\left(\sum_{T\in\mathcal{T}_h}h^{-3}_T\|Q_b
Q_0\omega-Q_b\omega\|^2_{\partial T}\right)^{\frac{1}{2}}\leq
Ch^{m-1}\|\omega\|_{m+1},
\end{eqnarray}
\begin{eqnarray}\label{ineqn7}
\left(\sum_{T\in\mathcal{T}_h}
\|\nabla(\Delta\omega)-Q_b(\nabla(\Delta\omega))\|^2_{\partial
T}\right)^{\frac{1}{2}}\leq Ch^{m-\frac32}\|\omega\|_{m+2}.
\end{eqnarray}
Here $\delta_{i,j}$ is the usual Kronecker's delta with value $1$
when $i=j$ and value 0 otherwise.
\end{lemma}

\begin{proof} To derive (\ref{ineqn3}), we use the trace inequality (\ref{Trace inequality00}) and the
estimate (\ref{ineqn2}) to obtain
\begin{eqnarray*}
&&\sum_{T\in\mathcal{T}_h} h_T\| \Delta \omega-\mathbb{Q}_h\Delta\omega\|^2_{\partial T}\\
&\leq &C\sum_{T\in\mathcal{T}_h}(\| \Delta \omega-\mathbb{Q}_h\Delta\omega\|^2_T+h^2_T\|\nabla(\Delta\omega-\mathbb{Q}_h\Delta\omega)\|^2_T)\\
&\leq &Ch^{2m-2}\|\omega\|^2_{m+1}.
\end{eqnarray*}
As to (\ref{ineqn4}), we use the trace inequality (\ref{Trace
inequality00}) and the estimate (\ref{ineqn2}) to obtain
\begin{eqnarray*}
&&\sum_{T\in\mathcal{T}_h} h^3_T\|\nabla( \Delta \omega-\mathbb{Q}_h\Delta\omega)\|^2_{\partial T}\\
&\leq &C\sum_{T\in\mathcal{T}_h}(h^2_T\| \nabla(\Delta
\omega-\mathbb{Q}_h\Delta\omega
)\|^2_T+h^4_T\|\nabla^2(\Delta\omega-\mathbb{Q}_h\Delta\omega)\|^2_T)\\
&\leq &Ch^{2m-2}(\|\omega\|^2_{m+1}+h^2\delta_{m,2}\|\omega\|^2_4).
\end{eqnarray*}
As to (\ref{ineqn5}), we have from the definition of $Q_b $, the
trace inequality (\ref{Trace inequality00}), and the estimate
(\ref{ineqn1}) that
\begin{eqnarray*}
&&\sum_{T\in\mathcal{T}_h} h^{-1}_T\|\nabla(Q_0\omega)\cdot\textbf{n}_e-Q_b(\nabla\omega\cdot\textbf{n}_e)\|^2_{\partial T}\\
&\leq&\sum_{T\in\mathcal{T}_h}h^{-1}_T\|(\nabla Q_0\omega-\nabla\omega)\cdot \textbf{n}_e\|^2_{\partial T}\\
&\leq & C\sum_{T\in\mathcal{T}_h}(h^{-2}_T\| \nabla
Q_0\omega-\nabla\omega
\|^2_T+\|\nabla Q_0\omega-\nabla\omega\|^2_{1,T})\\
&\leq & Ch^{2m-2}\|\omega\|^2_{m+1}.
\end{eqnarray*}

Notice that $Q_b$ is a linear bounded operator, we use the
definition of $Q_b$ and the trace inequality (\ref{Trace
inequality00}) to obtain
\begin{eqnarray*}
&&\sum_{T\in T_h}h_T^{-3}\|Q_b Q_0\omega-Q_b \omega\|_{\partial T}^2
\\
&\leq& \sum_{T\in T_h}(h_T^{-4}\|Q_0\omega-\omega\|_{T}^2+h_T^{-2}\|\nabla(Q_0\omega-\omega)\|_T^2)\\
&\leq& Ch^{2m-2}\|\omega\|_{m+1}^2.
\end{eqnarray*}
To derive (\ref{ineqn7}), we use the trace inequality (\ref{Trace
inequality00}) and the estimate (\ref{ineqn2}) to obtain
\begin{eqnarray*}
&&\sum_{T\in\mathcal{T}_h} \|\nabla(\Delta\omega)-Q_b(\nabla(\Delta\omega))\|^2_{\partial T}\\
&\leq &C\sum_{T\in\mathcal{T}_h}(h^{-1}_T\|
\nabla(\Delta\omega)-Q_b(\nabla(\Delta\omega))\|^2_T
+h_T\|\nabla(\nabla(\Delta\omega)-Q_b(\nabla(\Delta\omega)))\|^2_T)\\
&\leq &Ch^{2m-3}\|\omega\|^2_{m+2}.
\end{eqnarray*}
This completes the proof of (\ref{ineqn7}), and hence the lemma.
\end{proof}

\subsection{Technical inequalities}
The goal here is to present some technical estimates useful for
deriving error estimates for the WG finite element scheme
(\ref{WGalerkin_Algorithm}).

\begin{lemma}\label{Lemma:Lemma6.4} There exists a constant $C$ such
that, for any $v=\{v_0, v_b,v_n\textbf{n}_e\}\in V_h$, the following
holds true
\begin{eqnarray}\label{JWang-001}
\sum_{T\in \mathcal{T}_h}\|\Delta v_0\|^2_T\leq C|\!|\!| v
|\!|\!|^2.
\end{eqnarray}
\end{lemma}

\begin{proof}
From the identity (\ref{Discrete_wLaplacian-useful}) with
$\phi=\Delta v_0$ we have
\begin{eqnarray*}
\|\Delta v_0\|^2_T&=&(\Delta_w v, \Delta v_0)_T-\langle Q_b v_0-v_b,
\nabla(\Delta v_0)\cdot {\bf n} \rangle_{\partial T} +\langle
(\nabla v_0-v_n{\bf n}_e)\cdot {\bf n}, \Delta v_0 \rangle_{\partial
T}.
\end{eqnarray*}
Thus, using the Cauchy-Schwarz inequality, trace inequality, and the
inverse inequality we obtain
\begin{eqnarray*}
\|\Delta v_0\|^2_T
&\leq& \|\Delta_w v\|_T\|\Delta v_0\|_T+\|Q_bv_0-v_b\|_{\partial T}\|\nabla(\Delta v_0)\cdot{\bf n}\|_{\partial T}\\
&&+\|(\nabla v_0-v_n{\bf n}_e)\cdot{\bf n}\|_{\partial T}\| \Delta v_0\|_{\partial T}\\
&\leq& C (\|\Delta_w v\|_T\|\Delta v_0\|_T+h_T^{-\frac{1}{2}}\|Q_bv_0-v_b\|_{\partial T}\|\nabla(\Delta v_0)\cdot{\bf n}\|_{T}\\
&&+h_T^{-\frac{1}{2}}\|(\nabla v_0-v_n{\bf n}_e)\cdot{\bf
n}\|_{\partial T}\| \Delta v_0\|_{T})
\\
&\leq& C(\|\Delta_w v\|_T\|\Delta v_0\|_T+h_T^{-\frac{3}{2}}\|Q_bv_0-v_b\|_{\partial T}\|\Delta v_0\|_{T}\\
&&+h_T^{-\frac{1}{2}}\|(\nabla v_0-v_n{\bf n}_e)\cdot{\bf n}\|_{\partial T}\| \Delta v_0\|_{T}).\\
\end{eqnarray*}
Hence,
$$
\|\Delta v_0\|^2_T\leq C(\|\Delta_w
v\|_T^2+h_T^{-3}\|Q_bv_0-v_b\|_{\partial T}^2 +h_T^{-1}\|(\nabla
v_0-v_n{\bf n}_e)\cdot{\bf n}\|_{\partial T}^2),
$$
which verifies the inequality (\ref{JWang-001}).
\end{proof}

\begin{lemma}\label{DiscretePoincareinequality}
(\cite{CW-JW}, Lemma 10.4) There exists a constant $C$ such that,
for any $v\in V_h^0$, we have the following Poincar\'e inequality:
\begin{eqnarray}\label{DiscretePoincareinequality00}
\|v_0\|^2 \le C \left(\sum_{T\in\T_h}\|\nabla v_0\|_T^2 +h^{-1}
\sum_{T\in \mathcal{T}_h}\|Q_bv_0-v_b\|^2_{\partial T}\right).
\end{eqnarray}
\end{lemma}

The following lemma provides an estimate for the term
$\sum_{T\in\T_h}\|\nabla v_0\|_T^2$. Note that $v_0$ is a piecewise
polynomial of degree $k\ge 2$. Thus, Lemma \ref{Lemma:Lemma6.5new}
is concerned only with piecewise polynomials; no boundary condition
is necessary.

\begin{lemma}\label{Lemma:Lemma6.5new}
Let $\v$ be any piecewise polynomial of degree $k\ge 2$ on each
element $T$. Denote by $\nabla_h\v$ and $\Delta_h \v$ the gradient
and Laplacian of $\v$ taken on each element. Then, for any
$\varepsilon>0$, there exists a constant $C$ such that
\begin{eqnarray}\label{eqn:eqnLemma6.4new}
\nonumber\|\nabla_h \v\| ^2&\le& \varepsilon \|\v\|^2 +
C \varepsilon^{-1} \|\Delta_h \v\|^2\\
&&+ C \varepsilon^{-1}h^{-1} \left(\sum_{e\in
\mathcal{E}_h}\int_{e}\left(\frac{\partial \v_L}{\partial
\textbf{n}_L} + \frac{\partial \v_R}{\partial \textbf{n}_R}\right)^2
ds\right)
\\
&&\nonumber+ Ch^{-1}\left(\sum_{e\in \mathcal{E}_h}\int_{e} (Q_b
\v_R-Q_b \v_L)^2 ds\right).
\end{eqnarray}
Here $\v_L$ is the trace of $\v$ on $e$ as seen from the ``left" or
the opposite direction of $\bn_e$. If $e$ is a boundary edge, then
the trace from the outside of $\Omega$ is defined as zero.
\end{lemma}

\begin{proof}
On each element $T$, we have
\begin{eqnarray*}
\int_T |\nabla \v|^2 dT &=& - \int_T \v\Delta \v\  dT
+\int_{\partial T}\frac{\partial \v}{\partial \textbf{n}}\, \v \, ds
\\
&=& - \int_T \v\Delta \v dT +\int_{\partial T}\frac{\partial
\v}{\partial \textbf{n}}\, Q_b \v \, ds.
\end{eqnarray*}
Summing over all $T\in \mathcal{T}_h$, we have
\begin{eqnarray}\label{Eqn:EstimateEqn2New}
\|\nabla_h \v\|^2 = - \int_\Omega \v\Delta_h \v dT +\sum_{T\in
\mathcal{T}_h} \int_{\partial T}\frac{\partial \v}{\partial
\textbf{n}} \, Q_b \v \, ds.
\end{eqnarray}
Using the identity $a_L b_L+a_R b_R=(a_L+a_R)b_{L}+a_R(b_R-b_L)$ we
obtain
\begin{eqnarray*}
\sum_{T\in \mathcal{T}_h} \int_{\partial T}\frac{\partial
\v}{\partial \textbf{n}} \, Q_b \v \, ds&=& \sum_{e\in
\mathcal{E}_h}\int_{e}\left(\frac{\partial \v_L}{\partial
\textbf{n}_L} \, Q_b \v_L+ \frac{\partial \v_R}{\partial
\textbf{n}_R} \, Q_b \v_R\, \right) ds
\\
&=& \sum_{e\in \mathcal{E}_h}\int_{e}\left(\frac{\partial
\v_L}{\partial \textbf{n}_L}+ \frac{\partial \v_R}{\partial
\textbf{n}_R}\right)\, Q_b \v_L ds
\\
&&+ \sum_{e\in \mathcal{E}_h}\int_{e} \frac{\partial \v_R}{\partial
\textbf{n}_R}\, (Q_b \v_R-Q_b \v_L) ds.
\end{eqnarray*}
Thus, from the Cauchy-Schwarz inequality we have
\begin{eqnarray}\label{Eqn:EstimateEqn3New}
&&\left|\sum_{T\in \mathcal{T}_h} \int_{\partial T}\frac{\partial
\v}{\partial \textbf{n}} \, Q_b \v \, ds\right|\\
\nonumber&\le& \left(\sum_{e\in
\mathcal{E}_h}\int_{e}\left(\frac{\partial \v_L}{\partial
\textbf{n}_L} + \frac{\partial \v_R}{\partial \textbf{n}_R}\right)^2
ds\right)^{\frac12} \left(\sum_{e\in \mathcal{E}_h}\int_{e}\left|Q_b
\v_L \right|^2 ds\right)^{\frac12}
\\
\nonumber&&+ \left(\sum_{e\in
\mathcal{E}_h}\int_{e}\left|\frac{\partial \v_R}{\partial
\textbf{n}_R}\right|^2 ds\right)^{\frac12} \left(\sum_{e\in
\mathcal{E}_h}\int_{e} (Q_b \v_R-Q_b \v_L)^2 ds\right)^{\frac12}.
\end{eqnarray}
Next, we use the trace inequality (\ref{Trace inequality00}) and the
inverse inequality (\ref{Inverse Inequality00}) to obtain
\begin{eqnarray}\label{Eqn:EstimateEqn4New}
\int_{e}\left|Q_b \v_L \right|^2 ds &\le&
\int_{e}\left|\v_L \right|^2 ds\\
\nonumber&\le& C \left[h^{-1} \int_{T} \v^2 dT + h \int_{T} |\nabla
\v|^2 dT\right]
\\
\nonumber&\le& Ch^{-1} \int_{T} \v^2 dT,
\end{eqnarray}
and
\begin{eqnarray}\label{Eqn:EstimateEqn5New}
\int_{e}\left|\frac{\partial \v_R}{\partial \textbf{n}_R}\right|^2
ds &\le& C \left[h^{-1} \int_{T} |\nabla \v|^2 dT + h \int_{T}
|\nabla^2 \v|^2 dT\right]\\
\nonumber&\le& C h^{-1} \int_{T} |\nabla \v|^2 dT.
\end{eqnarray}
Substituting (\ref{Eqn:EstimateEqn4New}) and
(\ref{Eqn:EstimateEqn5New}) into (\ref{Eqn:EstimateEqn3New}) yields
\begin{eqnarray}\label{Eqn:EstimateEqn6New}
\left|\sum_{T\in \mathcal{T}_h} \int_{\partial T}\frac{\partial
\v}{\partial \textbf{n}} \, Q_b \v \, ds\right|&\le &C
h^{-\frac12}\|\v\| \left(\sum_{e\in
\mathcal{E}_h}\int_{e}\left(\frac{\partial \v_L}{\partial
\textbf{n}_L} + \frac{\partial \v_R}{\partial \textbf{n}_R}\right)^2
ds\right)^{\frac12}
\\
\nonumber&&+ C h^{-\frac12}\|\nabla_h \v\| \left(\sum_{e\in
\mathcal{E}_h}\int_{e} (Q_b \v_R-Q_b \v_L)^2 ds\right)^{\frac12}.
\end{eqnarray}
Substituting (\ref{Eqn:EstimateEqn6New}) into
(\ref{Eqn:EstimateEqn2New}) gives
\begin{eqnarray*}
\|\nabla_h \v\|^2&\le& \|\Delta_h \v\|\ \|\v\| +Ch^{-\frac12}\|\v\|
\left(\sum_{e\in \mathcal{E}_h}\int_{e}\left(\frac{\partial
\v_L}{\partial \textbf{n}_L} + \frac{\partial \v_R}{\partial
\textbf{n}_R}\right)^2 ds\right)^{\frac12}
\\
&&+ C h^{-\frac12}\|\nabla_h \v\|   \left(\sum_{e\in
\mathcal{E}_h}\int_{e} (Q_b \v_R-Q_b \v_L)^2 ds\right)^{\frac12},
\end{eqnarray*}
which, through an use of Young's inequality, implies the desired
estimate (\ref{eqn:eqnLemma6.4new}). This completes the proof.
\end{proof}

\begin{lemma}\label{DiscretePoincareinequality-2}
There exists a constant $C$ such that for any
$v=\{v_0,v_b,v_n\bn_e\}\in V_h^0$ the following Poincar\'e type
inequality holds true
\begin{eqnarray}\label{DiscretePoincareinequality-New}
\|\nabla_h v_0\| \le C \3bar v \3bar.
\end{eqnarray}
In addition, we have the following estimate
\begin{eqnarray}\label{JWang-Jan22.000}
\|\nabla_h v_0\| \le \lambda h^{-1} \|v\| + C h \3bar v\3bar,
\end{eqnarray}
where $\lambda$ is a positive constant.
\end{lemma}

\begin{proof}
The first component $v_0$ is a piecewise polynomial of degree $k\ge
2$. Using the estimate (\ref{eqn:eqnLemma6.4new}) in Lemma
\ref{Lemma:Lemma6.5new} we have
\begin{eqnarray}\label{eqn:eqnLemma:Jan20-001}
\nonumber\|\nabla_h v_0\| ^2&\le& \varepsilon \|v\|^2 +
C \varepsilon^{-1} \|\Delta_h v_0\|^2\\
&&+ C \varepsilon^{-1} h^{-1} \left(\sum_{e\in
\mathcal{E}_h}\int_{e}\left(\frac{\partial {v_0}_L}{\partial
\textbf{n}_L} + \frac{\partial {v_0}_R}{\partial
\textbf{n}_R}\right)^2 ds\right)
\\
&&\nonumber+ Ch^{-1}\left(\sum_{e\in \mathcal{E}_h}\int_{e} (Q_b
{v_0}_R-Q_b {v_0}_L)^2 ds\right).
\end{eqnarray}
By inserting $v_n\bn_e\cdot\bn$ in each integrand we obtain
$$
\sum_{e\in \mathcal{E}_h}\int_{e}\left(\frac{\partial
{v_0}_L}{\partial \textbf{n}_L} + \frac{\partial {v_0}_R}{\partial
\textbf{n}_R}\right)^2 ds \leq C \sum_{T\in\T_h} \|\nabla v_0
\cdot\bn_e - v_n\|_{\partial T}^2.
$$
Similarly, by inserting $v_b$
$$
\sum_{e\in \mathcal{E}_h}\int_{e} (Q_b {v_0}_R-Q_b {v_0}_L)^2 ds
\leq C \sum_{T\in\T_h} \|Q_b v_0 - v_b\|^2_{\partial T}.
$$
Substituting the above two inequalities into
(\ref{eqn:eqnLemma:Jan20-001}) yields
\begin{eqnarray}\label{JWang-Jan22.001}
\|\nabla_h v_0\| ^2&\le &\varepsilon \|v\|^2 + C \varepsilon^{-1}
\|\Delta_h v_0\|^2+ Ch^{-1}\sum_{T\in\T_h} \|Q_b v_0 -
v_b\|^2_{\partial T}\\
\nonumber&& + C \varepsilon^{-1} h^{-1}\sum_{T\in\T_h} \|\nabla v_0
\cdot\bn_e - v_n\|_{\partial T}^2.
\end{eqnarray}
Using the Poincar\'e inequality (\ref{DiscretePoincareinequality00})
and the estimate (\ref{JWang-001}) we arrive at
\begin{eqnarray}\label{eqn:eqnLemma:Jan20-002}
\nonumber\|\nabla_h v_0\| ^2\le \varepsilon C \|\nabla_h v\|^2 + C
\varepsilon^{-1} \3bar v \3bar^2,
\end{eqnarray}
which leads to the inequality (\ref{DiscretePoincareinequality-New})
for sufficiently small $\varepsilon$.

Finally, by setting $\varepsilon = \lambda h^{-2}$ in
(\ref{JWang-Jan22.001}) we arrive at
\begin{eqnarray}\nonumber \|\nabla_h v_0\| ^2\le \lambda h^{-2}
 \|v\|^2 + C h^2 \3bar v\3bar^2,
\end{eqnarray}
where $\lambda$ is a positive constant. This verifies the inequality
(\ref{JWang-Jan22.000}), and hence completes the proof of the lemma.
\end{proof}

\begin{lemma}\label{Lemma:Lemma6.5}
There exists a constant $C$ such that for any
$v=\{v_0,v_b,v_n\bn_e\}\in V_h^0$ one has
\begin{eqnarray}\label{eqn:eqnLemma6.4}
\sum_{T\in \mathcal{T}_h}\int_{\partial T}(v_0-Q_b v_0)^2 ds \leq C
h \3bar v \3bar^2
\end{eqnarray}
and
\begin{eqnarray}\label{eqn:eqnLemma6.4JW}
\sum_{T\in \mathcal{T}_h}\int_{\partial T}(v_0-Q_b v_0)^2 ds \leq C
\lambda h^{-1}\|v\|^2 + C h^3 \3bar v \3bar^2.
\end{eqnarray}
\end{lemma}

\begin{proof}
From the trace inequality (\ref{Trace inequality00}) and the inverse
inequality (\ref{Inverse Inequality00}), we have
\begin{eqnarray*}
\int_{\partial T}(v_0-Q_b v_0)^2 ds \le C h \int_T |\nabla v_0|^2
dT.
\end{eqnarray*}
Summing over all $T\in \mathcal{T}_h$ yields
\begin{eqnarray}\label{Eqn:EstimateEqn1}
\sum_{T\in \mathcal{T}_h}\int_{\partial T}(v_0-Q_b v_0)^2 ds \leq C
h\sum_{T\in \mathcal{T}_h} \int_T |\nabla v_0|^2 dT,
\end{eqnarray}
which, combined with (\ref{DiscretePoincareinequality-New}) and
(\ref{JWang-Jan22.000}), completes the proof of the lemma.
\end{proof}

\begin{remark}
The estimate (\ref{eqn:eqnLemma6.4}) in Lemma \ref{Lemma:Lemma6.5}
is sufficient for us to derive an optimal order error estimate for
the WG finite element solution arising from
(\ref{WGalerkin_Algorithm}). But the estimate
(\ref{eqn:eqnLemma6.4}) is sub-optimal in terms of the mesh
parameter $h$. We conjecture that the following inequality holds
true
\begin{eqnarray}\label{eqn:eqnLemma6.4optimal}
\sum_{T\in \mathcal{T}_h}\int_{\partial T}(v_0-Q_b v_0)^2 ds \leq C
h^3 \ \3bar v \3bar^2.
\end{eqnarray}
However, with the current mathematical approach, we are unable to
verify the validity of (\ref{eqn:eqnLemma6.4optimal}). This estimate
is then left to interested readers or researchers as an open
problem.
\end{remark}

\indent \textbf{Acknowledgements} We gratefully acknowledge
Professor Junping Wang for presenting this problem and giving us
many valuable suggestions. The authors also thank the anonymous
referees and editor for their careful reading of the manuscript and
their valuable comments to improvement the work.

\end{document}